\documentclass[11pt,twoside,reqno]{amsart}
\usepackage{amsmath, amsfonts, amsthm, amssymb,amscd, graphicx, amscd}
\usepackage{float,epsf}
\usepackage[english]{babel}
\usepackage{enumerate}
\usepackage{tikz}
\usepackage{hyperref}
 \usepackage{float} 
\usepackage{srcltx}
\usepackage{subfigure}
\usepackage{graphics}
\usepackage{epsfig}
\usepackage{geometry}
\usepackage{verbatim}
\usepackage{mathrsfs}

\renewcommand{\d}{\delta} 
 \renewcommand{\a}{\alpha}
 \renewcommand{\t}{\theta}
\newcommand{\s}{\sigma}

\newcommand{\mx}{\mu^\xi} 
\newcommand{\re}{\mathbb{R}}

\newtheorem{theorem}{Theorem}[section]

\newtheorem{proposition}[theorem]{Proposition}
\newtheorem{corollary}[theorem]{Corollary}
\newtheorem{remark}[theorem]{Remark}

\newtheorem{question}[theorem]{Question}

\numberwithin{equation}{section}
\numberwithin{figure}{section}

\makeatletter
\@namedef{subjclassname@2020}{Mathematics Subject Classification 2020}
\makeatother

\newcommand{\field}[1]{\mathbb{#1}}
\newcommand{\C}{\field{C}}

\def\intave#1{\int_{#1}\hbox{\llap{$\raise2.3pt\hbox{\vrule
height.9pt width7pt}\phantom{\scriptstyle{#1}}\mkern-2mu$}}}

\everymath{\displaystyle}

\title{flow approach on Riesz type nonlocal energies}
\usepackage[misc]{ifsym}
\author{Jiaxin He}
\address{School of Mathematics, Hunan University, Changsha, Hunan, China.}
\email{17773612125@163.com}

\author{Qinfeng Li}
\address{School of Mathematics, Hunan University, Changsha, P.R. China.}
\email{liqinfeng1989@gmail.com}

\author{Juncheng Wei}
\address{Department of Mathematics, The Chinese University of Hong Kong, Hong Kong.}
\email{wei@math.cuhk.edu.hk}

\author{Hang Yang}
\address{School of Mathematics, Hunan University, Changsha, P.R. China.}
\email{hangyang0925@gmail.com}

\thanks{Research of Qinfeng Li is supported by National Key R\&D Program of China (2022YFA1006900) and the National Science Fund of China General Program (No. 12471105).  The research of Juncheng Wei is partially supported
by GRF grant of HK RGC entitled “New frontiers in singular limits of elliptic and parabolic equations".}

\begin{document}
\maketitle

\begin{abstract}
Via continuous deformations based on natural flow evolutions, we prove several novel monotonicity results for Riesz-type nonlocal energies on triangles and quadrilaterals. Some of these results imply new and simpler proofs for known theorems without relying on any symmetrization arguments.

\end{abstract}

\section{Introduction}
Let $\Omega$ be a bounded domain in $\mathbb{R}^2$. Define
\begin{align*}
    D(\Omega)=\int_\Omega \int_\Omega K(|x-y|)\, dxdy,
\end{align*}where $K: \mathbb{R}^+\to \mathbb{R}$ is a $C^1$, strictly decreasing function and $\int_0^1 K(r)r\, dr<\infty$. The last condition is to guarantee that $D(\cdot)$ is finite on bounded planar domains, as shown in \cite{BCT}. $D(\Omega)$ is often called the Riesz-type potential nonlocal energies, and the prototype of $K(r)$ is given by $r^{-\alpha}, \, \alpha \in (0,2)$.

It is a well-established fact that, in any dimension, $D(\cdot)$ is uniquely maximized by a ball under a volume constraint, thanks to Riesz's rearrangement inequality \cite{Riesz}. The discrete optimization problem is then natural, and it is often conjectured that, among polygons with $N$ sides and fixed area, the regular polygon uniquely maximizes $D(\cdot)$. Surprisingly, this is in general not true, as it is proved in the very interesting work \cite{BBF} that the maximality of regular polygon should be sensitive to the choice of the kernel $K$, when $N$ is large. 

When $N=3,4$, the conjecture is indeed correct, as proved earlier in \cite{BCT} via an iteration of Steiner symmetrization, in the same flavor of the proof in \cite[Section 7.4]{PS51}. The rigidity results for triangles and quadrilaterals are also obtained in \cite{BCT} along some particularly chosen shape deformations.

We are motivated by the following question:

\begin{question}
\label{kq1}
Are there other area-preserving deformations, which are not based on the Steiner symmetrization, but along which $D(\cdot)$ is still increasing?
\end{question} 

This question is of significance since it can lead to new monotonicity results and provide more diverse optimization strategies. To tackle the question, we utilize the flow approach, as it not only naturally
handles continuous evolution and transforms the monotonicity investigation of $D(\cdot)$ into the sign
analysis of shape derivatives of $D(\cdot)$ along the flow, but also eliminates the reliance on Steiner symmetrization argument.

In this paper, we mainly focus our attention on triangles and quadrilaterals, which represent two important classes of planar domains. Though the geometries of triangles and quadrilaterals are simple, many significant results have been established on them, which often inspire further important results on generic shapes. For a comprehensive introduction on shape optimization problems on triangular domains, we refer to the classical monographs \cite{PS51} and \cite{Henrot}, as well as the excellent survey in \cite[Section 6]{nshap} by Laugesen and Siudeja.

Concerning Question \ref{kq1} for triangles, the first result we obtain is the following:

\begin{theorem}
\label{yangsheng2}
Let $\Omega=\triangle_{ABC}$ be a triangle, with $AB$ being the longest side and lying on the $x$-axis, and with $C$ lying on the positive $y$-axis. Suppose that $|AB|>|BC|\ge |AC|$, and we let $C_t=(1+t)C$, $\Omega_t=\sqrt{\tfrac{|OC|}{|OC_t|}}\triangle_{ABC_t}=\triangle_{A_tB_t\tilde{C}_t}$, where $O$ is the origin. Let $t_1>0$ be the unique number such that $|BA|=|BC_{t_1}|$. Then, for $t\in (-1,t_1)$, $D(\Omega_t)>0$ is a strictly increasing function. 
\end{theorem}

Similarly, we also have: \begin{theorem}
\label{yangsheng2'}
    Let $O$ be the origin, $\Omega=\triangle_{ABC}$ be a non-obtuse triangle with $AB$ being the shortest side and lying on the $x$-axis, and with $C$ lying on the positive $y$-axis. Suppose that $|AB|<|BC|\le  |AC|$, and we let $C_t=(1-t)C$ and $\Omega_t=\sqrt{|OC|/|OC_t|}\triangle_{ABC_t}$. Let $t_2>0$ be the unique number such that $|BA|=|BC_{t_2}|$.    Then, for $t\in (-\infty,t_2)$, $D(\Omega_t)$ is a strictly increasing function.
\end{theorem}

The motivation of establishing Theorems \ref{yangsheng2}-\ref{yangsheng2'} is as follows. A given non-equilateral triangle $\triangle_{ABC}$ must belong to one of the following two cases:
\begin{enumerate}
    \item It has exactly one shortest height.
    \item It has at least two shortest height. 
\end{enumerate} 
If $\triangle_{ABC}$ belongs to the first case, then we can continuously stretch the shortest height while scaling the triangle to keep its area constant, and during the process, the triangle gradually becomes an isosceles triangle belonging to the second case. Then, Theorem \ref{yangsheng2} says that along the evolution, $D(\cdot)$ is really strictly increasing. If $\triangle_{ABC}$ belongs to the second case, then it is a sub-equilateral triangle (i.e., an isosceles triangle with aperture less than or equal to $\pi/3$). Then Theorem \ref{yangsheng2'}, which is more general, gives the continuous deformation that gradually transforms such a triangle into an equilateral one by compressing the tallest height while preserving the area, thereby making $D(\cdot)$ increasing during the process.

An illustration of Theorem \ref{yangsheng2} and Theorem \ref{yangsheng2'} can be seen in Figures \ref{fig:ctm1}-\ref{fig:ctm2} below, and the proof of Theorems \ref{yangsheng2}-\ref{yangsheng2'} is by specifically constructing a time-dependent vector field $\eta$, which generates a flow map $F_t$ mapping $\Omega$ to $\Omega_t$, and then analyzing the sign of the derivative of $D(\Omega_t)$ via the reflection argument. 

\begin{figure}[htp]
\centering
\begin{tikzpicture}[scale = 2.5]

\fill (-0.4, 0) circle (0.02 ) node[below ] {\small$A$};
\fill (0,1) circle (0.02 ) node[right]{\small $C_{t_1}$};
\fill (0, 0) circle (0.02 ) node[below ] {\small$O$};
\fill (1.2, 0) circle (0.02 ) node[below ] {\small$B$};
\fill (0, 0.5) circle (0.02 ) node[right] {\small$C$};
\fill (0,0.7) circle (0.02 ) node[right] {\small$C_{t}$};

\draw[->] (-0.6,0)--(1.5,0) node[right] {$x$};
\draw[->] (0,-0.2)--(0,1.3) node[right] {$y$};

\draw (0,0)--(1.2,0);
\draw[thick,dotted] (0,0.7)--(1.2,0);
\draw[thick,dotted] (0,0.7)--(-0.4,0);
\draw[thick,dotted] (0,1)--(-0.4,0);
\draw[thick,dotted] (0,1)--(1.2,0);
\draw (-0.4,0)--(0,0.5);
\draw (1.2,0)--(0,0.5);

\node at (0.3,0.2) {\small $\Omega$};

\draw[->,thick] (1.2, 0.6)--(2.9,0.6);

\node[above] at (2.1, 0.7) {\small $\Omega_t=\sqrt{\tfrac{|OC|}{|OC_t|}}\triangle_{ABC_t}=\triangle_{A_tB_t\tilde{C}_t}$};


\draw[->] (2.9,0)--(5,0) node[right] {$x$};
\draw[->] (3.5,-0.2)--(3.5,1.3) node[right] {$y$};

\fill (3.5, 0) circle (0.02 ) node[below]{\small $O$};
\fill (3.5-0.845*0.4,0) circle (0.02 ) node[below left]{\small $A_t$};
\draw [thick, dotted] (3.5-0.845*0.4,0)--  (3.5, 0.7*0.845);
\draw [thick, dotted] (3.5-0.707*0.4,0)--  (3.5, 0.707);
\draw [thick, dotted] (3.5+0.845*1.2,0)--  (3.5, 0.7*0.845);
\draw [thick, dotted] (3.5+0.707*1.2,0)--  (3.5, 0.707);
\fill (3.5-0.707*0.4,0) circle (0.02 ) node[below ]{\small $A_{t_1}$};
\fill (3.5+1.2*0.845,0) circle (0.02 ) node[below]{\small $B_{t}$};
\fill (3.5+1.2*0.707,0) circle (0.02 ) node[below]{\small $B_{t_1}$};
\fill (3.5, 0.7*0.845) circle (0.02 ) node[below]{\small $\tilde{C}_{t}$};
\fill (3.5, 1*0.707) circle (0.02 ) node[right ]{\small $\tilde{C}_{t_1}$};

\end{tikzpicture}
\caption{An illustration of Theorem \ref{yangsheng2} for the case when $t\in (0,t_1]$. Stretching the shortest height $OC$ above while scaling to keep the area fixed. Then, $|\triangle_{A_tB_t\tilde{C}_t}|=|\triangle_{ABC}|$ and $D(\triangle_{A_tB_t\tilde{C}_t})$ is strictly increasing until $t=t_1$, at which time the triangle becomes isosceles.}
\label{fig:ctm1}
\end{figure}

\begin{figure}[htp]
\centering
\begin{tikzpicture}[scale = 2.5]

\fill (-0.8, 0) circle (0.02 ) node[below ] {\small$A$};
\fill (0,1.5) circle (0.02 ) node[right]{\small $C$};
\fill (0, 0) circle (0.02 ) node[below ] {\small$O$};
\fill (0.4, 0) circle (0.02 ) node[below ] {\small$B$};
\fill (0, 1.3) circle (0.02 ) node[right] {\small$C_t$};
\fill (0,1.1) circle (0.02 ) node[right] {\small$C_{t_2}$};

\draw[->] (-1.1,0)--(1,0) node[right] {$x$};
\draw[->] (0,-0.2)--(0,1.8) node[right] {$y$};

\draw (-0.8,0)--(0,1.5);
\draw[thick,dotted] (0,1.1)--(-0.8,0);
\draw[thick,dotted] (0,1.3)--(-0.8,0);
\draw[thick,dotted] (0,1.1)--(0.4,0);
\draw[thick,dotted] (0,1.3)--(0.4,0);
\draw (-0.8,0)--(0,1.5);
\draw (0.4,0)--(0,1.5);


\draw[->,thick] (1.0, 0.9)--(2.8,0.9);

\node[above] at (1.9, 1.0) {\small $\Omega_t=\sqrt{\tfrac{|OC|}{|OC_t|}}\triangle_{ABC_t}=\triangle_{A_tB_t\tilde{C}_{t}}$};


\fill (-0.8*1.074+4, 0) circle (0.02 ) node[below ] {\small$A_t$};
\fill (-0.8*1.168+4, 0) circle (0.02 ) node[below left] {\small$A_{t_2}$};
\fill (0.4*1.168+4, 0) circle (0.02 ) node[above] {\small$B_{t_2}$};
\fill (4, 0) circle (0.02 ) node[below ] {\small$O$};
\fill (4+0.4*1.074, 0) circle (0.02 ) node[below ] {\small$B_t$};
\fill (4, 1.074*1.3) circle (0.02 ) node[right] {\small$\tilde{C}_t$};
\fill (4,1.168*1.1) circle (0.02 ) node[right] {\small$\tilde{C}_{t_2}$};
\draw [thick, dotted](-0.8*1.074+4, 0) -- (4, 1.074*1.3);
\draw [thick, dotted](-0.8*1.168+4, 0) -- (4, 1.168*1.1);
\draw [thick, dotted](0.4*1.074+4, 0) -- (4, 1.074*1.3);
\draw [thick, dotted](0.4*1.168+4, 0) -- (4, 1.168*1.1);

\draw[->] (2.6,0)--(4.6,0) node[right] {$x$};
\draw[->] (4,-0.2)--(4,1.8) node[right] {$y$};

\end{tikzpicture}
\caption{An illustration of Theorem \ref{yangsheng2'} for the case when $0<t\le t_2$. Compressing below the tallest height $OC$ while scaling to keep the area fixed. Then, $|\triangle_{A_tB_t\tilde{C}_t}|=|\triangle_{ABC}|$, and $D(\triangle_{A_tB_t\tilde{C}_t})$ is strictly increasing until $t=t_2$, at which time the triangle becomes isosceles.}
\label{fig:ctm2}
\end{figure}

Therefore, combining the flow evolutions described in Theorem \ref{yangsheng2} and Theorem \ref{yangsheng2'}, we have found an area-preserving deformation path along which $D(\cdot)$ is strictly increasing, evolving any arbitrary non-equilateral triangle first into an isosceles triangle by stretching the shortest height, and then further into an equilateral one by compressing the tallest height. This answers Question \ref{kq1} for triangles, and also gives another yet simpler proof of the fact (see \cite[Theorem 1.1]{BCT}) that equilateral triangles uniquely maximize the Riesz-type nonlocal energy $D(\cdot)$, without the traditional limit process of step-by-step Steiner symmetrizations.

\vskip 0.3cm
A combination of Theorem \ref{yangsheng2} and Theorem \ref{yangsheng2'} immediately implies the following new monotonicity results on isosceles triangles:

\begin{corollary}
\label{coro}
    Let $I(\alpha)$ be the isosceles triangle with aperture $\alpha$ and given area. Then, $D(I(\alpha))$ is strictly increasing when $\alpha \in (0,\pi/3)$, and is strictly decreasing when $\alpha \in (\pi/3,\pi)$.
\end{corollary}

Similar results to Corollary \ref{coro} have been proved by Siudeja for Dirichlet boundary shape functionals, such as the torsional rigidity or first eigenvalue of Dirichlet Laplacian, see \cite{SB10}. His argument is via both the continuous Steiner symmetrization argument and the Steiner symmetrization argument. Such monotonicity results on nonlocal energies over isosceles triangles have not been established in the previous literature.

\vskip 0.2cm

The next result gives the strict decreasing property of $D(\cdot)$ when a triangle is strictly deviating from an isosceles one, by fixing an angle and the area while making the ratio of the longer leg to the shorter leg near that angle increasingly larger. The proof is motivated by that of Theorem \ref{yangsheng2}. Instead of continuously stretching the height, we can continuously stretch one leg while fixing the area, which generates another flow along which $D(\cdot)$ is monotone. The transformation is illustrated in Figure \ref{fig:ctm}.
\begin{theorem}
    \label{yangsheng1}
Let $\Omega_{\alpha,q}$ denote the triangle of a given area, with one angle equal to $\alpha$ and $q\ge 1$ being the ratio of the two legs near the angle. Then, for any fixed $\alpha \in (0,\pi)$, $D(\Omega_{\alpha,q})$ is a strictly decreasing function with respect to $q\ge 1$.
\end{theorem}

Similar type of result was also first proved by Siudeja \cite{SB10} on the first eigenvalue problem of the Dirichlet Laplacian, under the additional assumption that $\alpha$ is the smallest angle. The proof there is by the ingenious use of polar symmetrization. Later, by delicate applications of the continuous Steiner symmetrization argument, Solynin \cite{Solynin20} proves the strict monotonicity of the first eigenvalue without the smallest assumption on the fixed angle. The monotonicity of the Riesz-type nonlocal energy on $D(\Omega_{\alpha, q})$ has not been addressed before. Our proof via the flow method is also significantly different from previous ones used in \cite{SB10} and \cite{Solynin20}. Such ideas have also been recently applied by us to study Dirichlet shape optimization problems, see \cite{HLXY}. 

Note that Theorem \ref{yangsheng1} enables us to compare the magnitudes of $D(\cdot)$ on two triangles, if they have one same angle and the same area. It is then a natural question of comparing the magnitudes of $D(\cdot)$ on two triangles if they have one same side and the same area. Motivated by \cite{CHVY} and \cite{BCT}, we prove the following result, which gives another way of continuously deviating an isosceles triangle while $D(\cdot)$ is strictly decreasing during the deviation process. 
\begin{theorem}
\label{zhenxi1}
    Let $T_{l,q}$ denote the triangle of a given area, with one side equal to $l$ and $q\ge 1$ being the ratio of the other two sides. Then, for any fixed $l>0$, $D(T_{l,q})$ is a strictly decreasing function with respect to $q\ge 1$.
\end{theorem}

In summary, Theorem \ref{yangsheng2}, Theorem \ref{yangsheng1} and Theorem \ref{zhenxi1} together give three different ways of smoothly evolving an arbitrary non-isosceles triangle into an isosceles one, and during the evolutions, $D(\cdot)$ is always strictly increasing. 

\vskip 0.2cm

Next, we consider Question \ref{kq1} on quadrilaterals. It seems to us that the monotonicity of $D(\cdot)$ under continuous Steiner symmetrization has not been stated in the previous literature. Highly motivated by \cite{BCT} and \cite{CHVY}, we prove a slightly stronger result for a class of quadrilaterals: $D(\cdot)$ is strictly increasing under a partial continuous Steiner symmetrization, see Proposition \ref{T1} and Figure \ref{type1} in section 6. For the other quadrilaterals, the monotonicity under the partial symmetrization in Proposition \ref{T1} is in general not true. Nevertheless, the  monotonicity property of $D(\cdot)$ always holds for all quadrilaterals under the continuous Steiner symmetrization about the line perpendicular to one of the diagonals, as stated in Theorem \ref{quadritheorem} below.
 
\begin{theorem}
\label{quadritheorem}
    Let $\Omega_t$ be the continuous family of quadrilaterals with vertices $A_t=(x_A(1-t),y_A)$, $B=(a,0)$, $C_t=(x_C(1-t),y_C)$ and $D=(-a,0)$, where $a>0$, $x_A, x_C\in \mathbb{R}$ and $y_Ay_C<0$. If both $x_A$ and $x_C$ are not zero, then $D(\Omega_t)$ is a strictly increasing function for $0\le t\le 1$.
\end{theorem}

When the vertex $A=A_0$ changes to $A_1$ and the vertex $C=C_0$ changes to $C_1$ at time $t=1$, the quadrilateral turns into a kite. Then, we can similarly move the vertices $B$ and $D$ up or down to turn the kite into a rhombus. By Theorem \ref{quadritheorem}, $D(\cdot)$ is strictly increasing along the continuous deformation. Up until the point of deformation into the rhombus, the monotonicity of $D(\cdot)$ is essentially through the continuous Steiner symmetrization process. However, beginning with the rhombus, motivated by Theorem \ref{yangsheng2}, we employ another transformation as follows: we continuously compress the longer diagonal of the rhombus until the two diagonals are equal, while always scaling to fix the area during the deformation. This step does not involve any continuous symmetrization argument, and we still have the monotonicity result, as stated below.

\begin{theorem}
\label{rhombustheoremc}
    Let $\Omega_q$ be the rhombus with a given area, where $q\ge 1$ is the ratio of the lengths
 of the diagonals of the rhombus. Then, $D(\Omega_q)$ is a strictly decreasing function for $q\ge 1$.
\end{theorem}

 As a particular application of Theorem \ref{quadritheorem} and Theorem \ref{rhombustheoremc}, we also obtain another proof of the fact (see \cite[Theorem 1.1]{BCT}) that the squares uniquely maximize the Riesz-type nonlocal energy $D(\cdot)$, by finding a continuous deformation which evolves an arbitrary non-square quadrilateral into a square, and along the evolution $D(\cdot)$ is strictly increasing. Such continuous deformation rapidly transforms a quadrilateral into a square, eliminating the step of taking the limit of iterations of Steiner symmetrization. 

Applying a similar flow vector field, we also prove the following monotonicity result on rectangles, which is also not based on symmetrization, and is new to our knowledge.
\begin{theorem}
\label{recderipos}
    Let $R_q$ be the rectangle with a given area, where $q\ge 1$ is the ratio of the length
 and the width of the rectangle. Then, $D(R_q)$ is a strictly decreasing function for $q\ge 1$.
\end{theorem}

Theorem \ref{rhombustheoremc} and Theorem \ref{recderipos} also have their independent interests, since they serve as monotonicity results on rhombuses and rectangles. The proofs are similar to that of Theorem \ref{yangsheng2}.

\vskip 0.3cm
\noindent \textbf{Outline of the paper} In section 2, we derive the evolution equation for $D(\cdot)$ along a smooth flow, and we state some results on boundary properties of $V_\Omega(x):=\int_{\Omega}K(|x-y|)\, dy$, which will be crucial in the sign analysis of the derivative of $D(\cdot)$ along flow evolutions. In section 3, we will prove Theorem \ref{yangsheng2} and Theorem \ref{yangsheng2'}. In section 4, we will prove Theorem \ref{yangsheng1} and Theorem \ref{zhenxi1}. In section 5, we will prove Theorem \ref{rhombustheoremc} and Theorem \ref{recderipos}. In section 6, we prove Theorem \ref{quadritheorem}.

\section{Boundary comparison of $V_\Omega$ and evolution equation of $D(\cdot)$ along flows}

We first need the following proposition, which is crucial in the sign analysis of the derivative of $D(\cdot)$ along many of the flows utilized in the paper.

\begin{proposition}
\label{symmetrycomparison}
   Let $\Omega=\triangle_{ABC}$ be a triangle with $|BC|>|AC|$. Let $M$ be the midpoint of $AB$. For any $P\in BM$, let $P'\in AM$ such that $|PM|=|MP'|$. Then, $V_\Omega(P')>V_\Omega(P)$. 
\end{proposition}

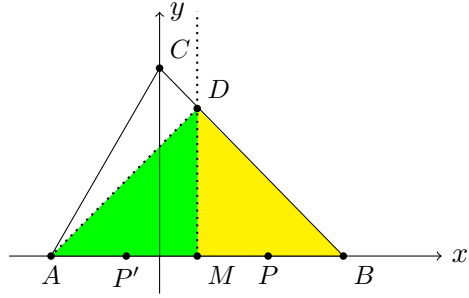
\begin{figure}[htp]
\centering
\begin{tikzpicture}[scale = 2.5]
\pgfmathsetmacro\L{1/sqrt(3)};
\pgfmathsetmacro\h{1};
\pgfmathsetmacro\t{0.4};

\draw[->] (-0.8,0)--(1.5,0) node[right] {$x$};
\fill[gray, yellow] (0.2, 0) -- (0.2, 0.785365) -- (\L+\t, 0) -- cycle;
\fill[gray, green] (0.2, 0) -- (0.2, 0.785365) -- (-\L, 0) -- cycle;

\draw (-\L,0)--(\L+\t,0);
\draw (\L+\t,0)--(0,\h);
\draw (0,\h)--(-\L,0);
\draw[thick,dotted] (0.2,0)--(0.2,1.3);
\draw[thick,dotted] (0.2,0.785365)--(-\L,0);

\draw[->] (0,-0.2)--(0,1.3) node[right] {$y$};

\fill (-\L, 0) circle (0.02 ) node[below ] {\small$A$};
\fill (-\L+\t, 0) circle (0.02 ) node[below ] {\small$P'$};
\fill (\L, 0) circle (0.02 ) node[below ] {\small$P$};
\fill (\L+\t, 0) circle (0.02 ) node[below right] {\small$B$};
\fill (0, \h) circle (0.02 ) node[above right] {\small$C$};
\fill (0.2, 0) circle (0.02 ) node[below right] {\small$M$};
\fill (0.2, 0.785365) circle (0.02 ) node[above right] {\small$D$};


\end{tikzpicture}
\caption{If $|BC|>|AC|$, then $V_\Omega(P)<V_\Omega(P')$.}
\label{fig:ctm'}
\end{figure}

\begin{proof}
    First, by symmetry, we have 
    \begin{align*}
\int_{\triangle_{ABD}}K(|PQ)|\,dQ=\int_{\triangle_{ABD}}K(|P'Q)|\, dQ.
    \end{align*}
Then, for any $Q\in int(\triangle_{ACD})$, $|PQ|>|P'Q|$, and hence by the assumption of $K$, we have
\begin{align*}
    \int_{\triangle_{ACD}}K(|PQ|)\, dQ<\int_{\triangle_{ACD}}K(|P'Q|)\, dQ.
\end{align*}
Therefore,
\begin{align*}
    V_\Omega(P)=&\int_{\triangle_{ABD}}K(|PQ)|\,dQ+\int_{\triangle_{ACD}}K(|PQ|)\, dQ\\
    <& \int_{\triangle_{ABD}}K(|P'Q)|\, dQ+\int_{\triangle_{ACD}}K(|P'Q|)\, dQ=V_\Omega(P').
\end{align*}
\end{proof}

Similarly, we have
\begin{proposition}
    \label{symmetrycomparison2}
   Let $\Omega=\triangle_{ABC}$ be a triangle with $|AB|>|AC|$. For any $P\in AC,\, P'\in AB$ with $|PA|=|P'A|$, we have $V_\Omega(P')>V_\Omega(P)$.
\end{proposition}

\begin{figure}[htp]
\centering
\begin{tikzpicture}[scale=2]
\pgfmathsetmacro\jiaoA{70};
\pgfmathsetmacro\LenB{4.0000};
\pgfmathsetmacro\LenC{\LenB*0.4000};
\pgfmathsetmacro\Cx{\LenC*cos(\jiaoA)};
\pgfmathsetmacro\Cy{\LenC*sin(\jiaoA)};
\pgfmathsetmacro\px{0.6*\Cx};
\pgfmathsetmacro\py{0.6*\Cy};
\pgfmathsetmacro\qx{0.6*\LenC};
\pgfmathsetmacro\qy{0};
\pgfmathsetmacro\Qx{\LenC*sin(\jiaoA)*\LenB/((\LenB-\LenC*cos(\jiaoA))*(tan(\jiaoA/2)+\LenC*sin(\jiaoA)/(\LenB-\LenC*cos(\jiaoA))))};
\pgfmathsetmacro\Qy{tan(\jiaoA/2)*\Qx};
\fill[gray, yellow] (0, 0) -- (\Qx, \Qy) -- (\LenC, 0) -- cycle;
\fill[gray, green] (0, 0) -- (\Qx, \Qy) -- (\Cx, \Cy) -- cycle;
\draw[black, dashed] (\Qx, \Qy) -- (\LenC, 0);
\draw[thick, red] (0, 0) -- (\Qx, \Qy);
\fill (\px, \py) circle (0.02) node[left] {\small $P$};
\fill (\qx, \qy) circle (0.02) node[below] {\small $P'$};
\draw[dotted, thick] (\px, \py)--(\qx, \qy);
\node[below left] at (0, 0) {$A$};
\node[below right] at (\LenB, 0) {$B$};
\node[left] at (\Cx, \Cy) {$C$};
\fill (\LenC, 0) circle (0.02) node[below] {\small ${C'}$};
\fill (\Qx, \Qy) circle (0.02) node[above right] {\small ${Q}$};
\draw[thick] (0, 0)--(\LenB, 0)--(\Cx, \Cy)--cycle;
\end{tikzpicture}
\caption{If $|AC|<|AB|$, then $V_\Omega(P)<V_\Omega(P')$.}
\label{fig21}
\end{figure}
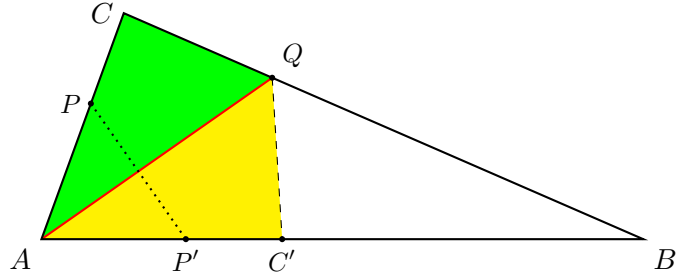

\begin{proof}
    Select $Q\in BC$ such that $AQ$ bisects $\angle A$. Let $C' \in AB$ be the reflection point of $C$ about the line segment $AQ$. For any point $Z \in \triangle_{BC'Q}$, since $|PZ|>|P'Z|$, we have
    \begin{align*}
        \int_{\triangle_{BC'Q}}K(|PZ|)\, dZ<\int_{\triangle_{BC'Q}}K(|P'Z|)\, dZ.
    \end{align*}
Let $D=\triangle_{ACQ}\cup \triangle_{AC'Q}$. Then by symmetry, we have
\begin{align*}
\int_{D}K(|PZ|)\,dZ=\int_{D}K(|P'Z|)\, dZ.
    \end{align*}
Adding them up, we thus obtain $V_\Omega(P)<V_\Omega(P')$.    
\end{proof}

The next theorem gives the evolution equation for the Riesz-type nonlocal energy $D(\cdot)$ along a flow map generated by a smooth time-dependent vector field. Classical shape derivative formulas are often presented at a particular time, say $t=0$, see for example \cite[Proposition 3.3]{BCT}, while the whole-time evolution equation along a flow is rarely seen in literature and will be of great importance to study the monotonicity properties of $D(\cdot)$ along flow evolutions. Therefore, we also include a proof of the derivation of the evolution equation, which holds in any dimension.

\begin{theorem}
\label{Tshapederivative}
    Let $\Omega$ be a piecewise smooth domain in $\mathbb{R}^n$, $\eta^t(x):=\eta(t,x): \mathbb R_+\times \mathbb{R}^n \to \mathbb{R}^n$ be a smooth vector field, and $F_t$ be the flow map generated by $\eta$. That is,
    \begin{align*}
        \begin{cases}
            \frac{\partial}{\partial t}F_t(x)=\eta(t,F_t(x))\quad t>0\\
            F_0(x)=x
        \end{cases}
    \end{align*}
Let $\Omega_t=F_t(\Omega)$. Then for any $t>0$,
\begin{align*}
    \frac{d}{dt}D(\Omega_t)=2\int_{\partial \Omega_t}V_{\Omega_t}(x) (\eta^t\cdot \nu) \, d\sigma,
\end{align*}
where $\nu$ is the outward unit normal to the boundary.
\end{theorem}

\begin{proof}
Let $JF_t=det(\nabla F_t)$. Then by Jacobi's matrix determinant formula and the chain rule, we have
\begin{align*}
    \frac{d}{dt}JF_t=&JF_t \, tr\left((\nabla F_t)^{-1} \frac{d}{dt} \nabla F_t\right)\\
    =& JF_t  tr\left((\nabla F_t)^{-1}  \nabla (\frac{d}{dt}F_t)\right)\\
     =& JF_t  tr\left((\nabla F_t)^{-1}  \nabla (\eta^t\circ F_t)\right)\\
    =&JF_t tr \left((\nabla F_t)^{-1} (\nabla \eta^t \circ F_t) \nabla F_t\right)\\
    =& (\text{div}\eta^t \circ F_t)JF_t.
\end{align*}
    By the area formula,
    \begin{align*}
      D(\Omega_t)=\int_\Omega \int_\Omega K(|F_t(x)-F_t(y)|)JF_t(x)JF_t(y)\, dxdy.
    \end{align*}
Hence, up to a regularization of the kernel $K$, (see for example the argument of the proof of \cite[Proposition 3.3]{BCT}), we can take the time derivative inside the integral, and the following calculation is valid.
\begin{align*}
    \frac{d}{dt}D(\Omega_t)=&\int_\Omega\int_\Omega K'(|F_t(x)-F_t(y)|)\frac{F_t(x)-F_t(y)}{|F_t(x)-F_t(y)|}\cdot \left(\eta^t(F_t(x))-\eta^t(F_t(y))\right) JF_t(x)JF_t(y)\, dxdy\\
    &+2\int_\Omega \int_\Omega K(|F_t(x)-F_t(y)|)\text{div}\eta^t(F_t(x)) JF_t(x)JF_t(y)\, dxdy\\
    =&\int_{\Omega_t}\int_{\Omega_t}K'(|x-y|)\frac{x-y}{|x-y|}\left(\eta^t(x)-\eta^t(y)\right)\, dxdy+2\int_{\Omega_t}\int_{\Omega_t}K(|x-y|)\text{div}\eta^t \, dxdy\\
    =&2\int_{\Omega_t}\int_{\partial \Omega_t}K(|x-y|)\eta^t(x)\cdot \nu (x)\, d\sigma_x dy,\\
    &\quad \mbox{by applying the divergence theorem on the second term}\\
    =&2\int_{\partial \Omega_t}V_{\Omega_t}(x) (\eta^t \cdot \nu) \, d\sigma,\quad \mbox{by Fubini's Theorem.}
\end{align*}
\end{proof}

If we choose $F_t$ to be a translation flow, then we immediately obtain from Theorem \ref{Tshapederivative} and the translating invariance of $D(\cdot)$ that, for a triangle $\Omega=\triangle_{ABC}$, 
\begin{align*}
     \frac{1}{|AB|}\int_{AB}V_\Omega\, ds=\frac{1}{|BC|}\int_{BC}V_\Omega\, ds=\frac{1}{|AC|}\int_{AC}V_\Omega\, ds.
\end{align*}
This equality is also derived in \cite{BCT} through the parallel movement of a side, inspired by \cite{FV19}. This, combined with Proposition \ref{symmetrycomparison2}, reveals an intriguing geometric phenomenon: in an arbitrary triangular domain $\Omega$, the longer the side length, the larger the maximum value of $V_\Omega$ on that side. However, regardless of the side length, the averaged value of $V_\Omega$ on each side remains the same. Then, a natural question arises: what about the case of rectangles?

Note that similar to the proof of Proposition \ref{symmetrycomparison2}, we know that if $\Omega=\square_{ABCD}$ is a rectangle with the condition $|AB|>|AD|$, then 
\begin{align*}
     \Vert V_\Omega \Vert_{L^\infty(AB)}> \Vert V_\Omega \Vert_{L^\infty(AD)}.
\end{align*}
Regarding the averaged value of $V_\Omega$ on each side of the rectangle, quite different from the triangular case, we have:

\begin{proposition}
\label{not-equidistribution-thmintro}
Let $\Omega=\square_{ABCD}$ be a rectangle with $|AB|>|AD|$. Then, 
\begin{align}
\label{rectangle-neumann-data}
\frac{\int_{AB}V_\Omega\,ds}{|AB|}> \frac{\int_{AD}V_\Omega\,ds}{|AD|}.
\end{align} 
\end{proposition}

This proposition turns out to be crucial to prove Theorem \ref{recderipos}.

\begin{proof}[Proof of Proposition \ref{not-equidistribution-thmintro}]
Without loss of generality, we assume that $A=(0,0)$, $B=(a,0)$, $C=(a,b)$ and $D=(0,b)$, with $a>b$. Let $E$ be the point of the intersection of the line $l: \, y=x$ and the side $CD$. Let $M$ be the midpoint of $AD$, $N$ be the midpoint of $AB$ and $M'\in AN$ with $|AM'|=|AM|$. See Figure \ref{fig-rectangle-4} below. 

Since the reflection of $DE$ about the line $l$ strictly lies inside the rectangle $\square ABCD$, for any $P\in AM$, $P'\in AM'$ with $|AP|=|AP'|$, similar to the proof of Proposition \ref{symmetrycomparison2}, we have $V_\Omega(P)<V_\Omega(P')$. By reflection argument and also similar to the proof of Proposition \ref{symmetrycomparison}, we have that $V_\Omega$ is strictly increasing along the segment $AN$ from $A$ to $N$. Therefore,
\begin{align*}
    \frac{1}{|AM|}\int_{AM}V_\Omega\, ds<\frac{1}{|AM'|}\int_{AM'}V_\Omega\, ds<\frac{1}{|M'N|}\int_{M'N}V_\Omega\, ds.
\end{align*}
Therefore,
\begin{align*}
    \frac{1}{|AN|}\int_{AN}V_\Omega\, ds=&\frac{1}{|AM'|+|M'N|}\left(\int_{AM'}V_\Omega\, ds+\int_{M'N}V_\Omega\, ds\right)\\
    >&\frac{1}{|AM|}\int_{AM}V_\Omega\, ds
\end{align*}
By symmetry, 
\begin{align*}
    \frac{1}{|AN|}\int_{AN}V_\Omega\, ds= \frac{1}{|AB|}\int_{AB}V_\Omega\, ds,\quad  \frac{1}{|AM|}\int_{AM}V_\Omega\, ds= \frac{1}{|AD|}\int_{AD}V_\Omega\, ds.
\end{align*}
Hence \eqref{rectangle-neumann-data} is obtained.
\end{proof}

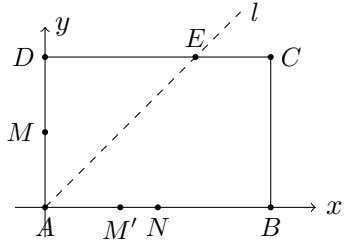
\begin{figure}[htp]
\centering
\begin{tikzpicture}[scale = 2]
\pgfmathsetmacro\Ls{2};
\pgfmathsetmacro\hs{1};
\pgfmathsetmacro\ts{0.3};
\pgfmathsetmacro\hts{\hs*(1+\ts)};

\fill (0, 0) circle (0.02 ) node[below] {\small $A$};
\fill (0, 1) circle (0.02 ) node[left] {\small$D$};
\fill (1.5,0) circle (0.02) node[below] {\small$B$};
\fill (1.5, 1) circle (0.02) node[right] {\small$C$};
\fill (0,0.5) circle (0.02) node[left]{\small $M$};
\fill (0.5,0) circle (0.02) node[below]{\small $M'$};
\fill (1,1) circle (0.02) node[above]{\small $E$};
\fill (0.75,0) circle (0.02) node[below]{\small $N$};

\draw[dashed] (0,0)--(1.3,1.3);

\fill (1.3,1.3) node[right]{\small $l$};

\draw[->] (-0.2,0)--(1.8,0) node[right] {$x$};
\draw[->](0,-0.2)--(0,1.2) node[right] {$y$};
\draw (0,1)--(1.5,1);
\draw  (1.5,0)--(1.5,1);

\end{tikzpicture}
\caption{Picture illustration of the proof of Proposition \ref{not-equidistribution-thmintro}}
\label{fig-rectangle-4}
\end{figure}

\begin{remark}
    The boundary behavior of $V_\Omega$ closely resembles that of the squared norm of the gradient of the torsion function, as evidenced by \cite{HLXY} and \cite{4L24}. The location of the maximum norm of the gradient of the torsion function is an important problem and has been under extensive research, see \cite{4L24} and references therein. In particular, \cite{4L24} demonstrates that for a triangular domain, the maximum of the norm of the gradient of the torsion function uniquely occurs on the longest side, between the midpoint and the foot of the altitude. Furthermore, for nearly equilateral triangles, it is shown that the point of maximum norm is closer to the midpoint. It would be an interesting question itself to derive similar results for $V_\Omega$.
\end{remark}

\section{Monotonicity of $D(\cdot)$ on triangles via height stretching and compressing flows}

In this section, we will prove Theorem \ref{yangsheng2} and Theorem \ref{yangsheng2'}.

\begin{proof}[Proof of Theorem \ref{yangsheng2}]
Without loss of generality, we assume $C=(0,1)$.  Let $$F_t(x,y)=\frac{1}{\sqrt{1+t}}\left(x, (1+t)y\right).$$
Then $F_t(\Omega)=\Omega_t=\triangle_{A_tB_t\tilde{C}_t}$. We construct the following height-stretching time-dependent vector field 
\begin{align*}
    \eta(t,x,y)=\frac{1}{2}\frac{1}{1+t}\left(-x,y\right).
\end{align*}
Since 
\begin{align*}
    \frac{\partial}{\partial t}F_t(x,y)=\frac{1}{2}(1+t)^{-3/2}\left(-x,(1+t)y\right),
\end{align*}
we have 
\begin{align*}
    \frac{\partial }{\partial t}F_t(x,y)=\eta\left(t,F_t(x,y)\right).
\end{align*}
That is, $\eta$ is exactly the smooth vector field generating the flow map $F_t$. Hence by Theorem \ref{Tshapederivative}, we have
\begin{align}
\label{11}
    \frac{d}{dt}D(\Omega_t)=\int_{\partial \Omega_t}\frac{1}{1+t}V_{\Omega_t}((x,y)) (-x,y)\cdot \nu\, ds.
\end{align}
Let $\beta_t=\angle A_tB_t\tilde{C}_t$ and $\gamma_t=\angle B_tA_t\tilde{C}_t$. On $A_tB_t$, 
\begin{align*}
    (-x,y)\cdot \nu=(-x,y)\cdot (0,-1)=-y=0.
\end{align*}
On $B_t\tilde{C}_t$,
\begin{align*}
      (-x,y)\cdot \nu=(-x,y)\cdot (\sin\beta_t,\cos\beta_t)=\cos \beta_t (y-x\tan\beta_t).
\end{align*}
On $A_t\tilde{C}_t$,
\begin{align*}
    (-x,y)\cdot \nu=(-x,y)\cdot (-\sin\gamma_t,\cos\gamma_t)=\cos \gamma_t (y+x \tan \gamma_t).
\end{align*}
Since $\partial \Omega_t=A_tB_t\cup B_t\tilde{C}_t\cup A_t\tilde{C_t}$, by \eqref{11} and the above, we have
\begin{align*}
  &\frac{d}{dt}D(\Omega_t)\\
     =&\frac{\cos \beta_t}{1+t}\int_{B_t\tilde{C}_t} V_{\Omega_t}((x,y)) (y-x\tan \beta_t) \, ds+\frac{\cos \gamma_t}{1+t}\int_{A_t\tilde{C}_t} V_{\Omega_t}((x,y)) (y+x\tan \gamma_t) \, ds.
\end{align*}
Let $M_t$ be the midpoint of $B_t\tilde{C}_t$. Since $y=(\tan \beta_t)x$ is exactly the line passing through the origin and $M_t$, for any point $p_t=(x,y)\in B_tM_t$ and $p_t'=(x',y')\in M_t\tilde{C}_t$ with $|p_tM_t|=|p_t'M_t|$, we have 
\begin{align*}
    0>y-x\tan\beta_t=-(y'-x'\tan\beta_t).
\end{align*}
Also, for $-1<t<t_1$, $|A_tB_t|>|A_t\tilde{C}_t|$, and thus by Proposition \ref{symmetrycomparison}, we have $V_{\Omega_t}(p_t)<V_{\Omega_t}(p_t')$.
Hence 
\begin{align*}
    \int_{B_t\tilde{C}_t} V_{\Omega_t}((x,y)) (y-x\tan \beta_t) \, ds=\int_{B_tM_t}\left(V_{\Omega_t}(p_t)-V_{\Omega_t}(p_t')\right)(y-x\tan \beta_t)\, ds>0.
\end{align*}
Similarly, let $N_t$ be the midpoint of side $A_t\tilde{C}_t$. Since $y=-(\tan \gamma_t)x$ is the line passing through the origin and $N_t$, for any point $q_t=(x,y)\in A_tN_t$ and $q_t'=(x',y')\in N_t\tilde{C}_t$ with $|q_tN_t|=|q_t'N_t|$, we have 
\begin{align*}
    0>y+x\tan\gamma_t=-(y'+x'\tan\gamma_t).
\end{align*}
Again by Proposition \ref{symmetrycomparison}, for any $-1<t<t_1$, we have
\begin{align*}
    \int_{A_t\tilde{C}_t} V_{\Omega_t}((x,y)) (y+x\tan \gamma_t) \, ds=\int_{A_tN_t}\left(V_{\Omega_t}(q_t)-V_{\Omega_t}(q_t')\right)(y+x\tan \gamma_t)\, ds>0.
\end{align*}
Hence $\tfrac{d}{dt}D(\Omega_t)>0$ for any $t\in (-1,t_1)$. This finishes the proof.
\end{proof}

The proof above is illustrated in Figure \ref{fig:add1} below. 

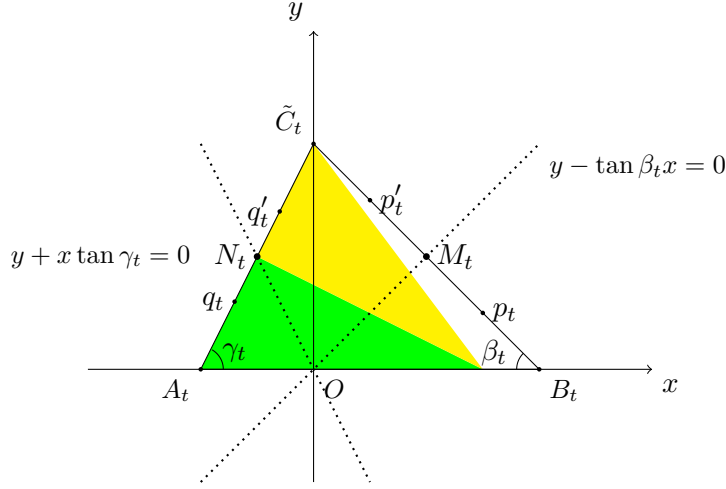
\begin{figure}[htp]
\centering
\begin{tikzpicture}[scale = 1.5]
\pgfmathsetmacro\d{180/pi};

\fill[green,opacity=0.4] (-1,0)--(3/2,0)--(-1/2,1)--cycle;
\fill[yellow,opacity=0.4] (0,2)--(3/2,0)--(-1/2,1)--cycle;
\draw (-1,0)--(0,2); \draw (0,2)--(2,0);\draw (-1,0)--(2,0);

\draw[->] (-2,0)--(3,0) node[below right] {$x$};
\draw[->] (0,-1)--(0,3) node[above left] {$y$};

\node[below left] at (-1,0) {\small $A_t$};
\node[above left] at (0,2) {\small $\tilde{C}_t$};
\node[below right] at (2,0) {\small $B_t$};
\node[below right] at (0,0) {\small $O$};

\fill (-1,0) circle (0.02);
\fill (0,2) circle (0.02);
\fill (2,0) circle (0.02);

\draw[samples=500,black, domain=0:pi/3, variable=\t] plot ({-1+0.2*cos(\t*\d)},{0.2*sin(\t*\d)});
\node at (-0.7,0.15) {$\gamma_{t}$};

\draw[samples=500,black, domain=15*pi/20:pi, variable=\t] plot ({2+0.2*cos(\t*\d)},{0.2*sin(\t*\d)});
\node at (1.6,0.15) {$\beta_{t}$};

\fill (1,1) circle (0.03) node[right] {$M_{t}$};
\draw[samples=500,thick,dotted, domain=-1:2, variable=\t] plot ({\t},{1*\t});
\node[right] at (2,1.8) {\small $y-\tan \beta_{t}x=0$};

\fill (-1/2,1) circle (0.03) node[left] {$N_{t}$};
\draw[samples=500,thick,dotted, domain=-1:0.5, variable=\t] plot ({\t},{-2*\t});
\node[left] at (-1,1) {\small $y+x\tan \gamma_{t}=0$};

\fill (0.5,1.5) circle (0.02) node[right] {$p_{t}'$};
\fill (1.5,.5) circle (0.02) node[right] {$p_{t}$};

\fill (-0.3,1.4) circle (0.02) node[left] {$q_{t}'$};
\fill (-0.7,0.6) circle (0.02) node[left] {$q_{t}$};

\end{tikzpicture}
\caption{For $-1<t<t_1$, $|A_tB_t|>|B_t\tilde{C}_t|\ge |A_t\tilde{C}_t|$. $y=x\tan \beta_t$ and $y=-x\tan \gamma_t$ pass through the midpoints of $B_t\tilde{C}_t$ and $A_t\tilde{C}_t$, respectively. $|p_tM_t|=|p_t'M_t|$, $V_{\Omega_t}(p_t)<V_{\Omega_t}(p_t')$, $|q_tN_t|=|q_t'N_t|$, $V_{\Omega_t}(q_t)<V_{\Omega_t}(q_t')$.}
\label{fig:add1}
\end{figure}

Next, we prove Theorem \ref{yangsheng2'}.

\begin{proof}[Proof of Theorem \ref{yangsheng2'}]
    Without loss of generality, we assume that $|OC|=1$. Let 
    \begin{align*}
        F_t(x,y)=\frac{1}{\sqrt{1-t}}\left(x,(1-t)y\right),\quad \eta(t,x,y)=\frac{1}{2(1-t)}(x,-y).
    \end{align*}
Then $\Omega_t=F_t(\Omega)$, and one can check that $F_t$ is generated by the vector field $\eta$.

By the similar derivation as in the proof of Theorem \ref{yangsheng2}, we have
\begin{align*}
    \frac{d}{dt}T(\Omega_t)=&\frac{\cos \beta_t}{1-t}\int_{B_t\tilde{C}_t} V_{\Omega_t}((x,y))(x\tan\beta_t-y)\, ds\\
    &+\frac{\cos \gamma_t}{1-t}\int_{ A_t\tilde{C}_t} V_{\Omega_t}((x,y))(-x\tan\gamma_t-y)\, ds,
\end{align*}
where $\beta_t=\angle \tilde{C}_tB_tA_t$ and $\gamma_t=\angle \tilde{C}_t A_tB_t$. The similar argument in the proof of Theorem \ref{yangsheng2} also ensures that $\tfrac{d}{dt}D(\Omega_t)>0$, for any $t\in (-\infty, t_2)$. 
\end{proof}

To the end, we remark that in Theorem \ref{yangsheng2}, if $|AB|>|BC|=|AC|$, then the strict increasing property of $D(\cdot)$ during the shortest height stretching process implies that when $\alpha\in (\pi/3,\pi)$, $D(I(\alpha))$ is a strictly decreasing function. Similarly, if in Theorem \ref{yangsheng2'}, $|AB|<|BC|=|AC|$, then it implies that when $\alpha\in (0,\pi/3)$, $D(I(\alpha))$ is a strictly increasing function. Therefore, we immediately have our Corollary \ref{coro}.

\section{Monotonicity of $D(\cdot)$ on triangles via leg stretching and parallel movement flow}

The main goal of this section is to prove Theorem \ref{yangsheng1} and Theorem \ref{zhenxi1}.

To prove Theorem \ref{yangsheng1}, it suffices to prove the following theorem: 
\begin{theorem}
\label{yangsheng1'}
Let $\Omega$ be a triangle $\triangle_{ABC}$, where $A$ lies at the origin, $B$ lies on the positive $x$-axis and $|AB|=|AC|$. Let $B_t=(1+t)B$ and $\Omega_t=\sqrt{|AB|/|AB_t|}\triangle_{AB_tC}=\triangle_{A\tilde{B}_tC_t}$. Then, for $t\in [0,+\infty)$, $D(\Omega_t)$ is a strictly decreasing function. 
\end{theorem}

\begin{figure}[htp]
\centering
\begin{tikzpicture}[scale = 2.5]

\fill (0, 0) circle (0.02 ) node[below ] {\small$A$};
\fill (1, 0) circle (0.02 ) node[below ] {\small$B$};
\fill (0.707, 0.707) circle (0.02 ) node[above right] {\small$C$};
\fill (1.2, 0) circle (0.02 ) node[below] {\small$B_{t}$};

\draw[->] (-0.3,0)--(1.5,0) node[right] {$x$};
\draw[->] (0,-0.2)--(0,1.1) node[right] {$y$};

\draw (0,0)--(1.2,0);
\draw[thick,dotted] (1,0)--(0.707,0.707);
\draw (0.707,0.707)--(1.2,0);
\draw (0.707,0.707)--(0,0);

\node at (0.5,0.2) {\small $\Omega$};

\draw[->,thick] (1.6, 0.4)--(2.6,0.4);

\node[above] at (2.1, 0.5) {\small $\Omega_t=\sqrt{\tfrac{|AB|}{|AB_t|}}\triangle_{AB_tC}$};


\draw[->] (2.7,0)--(4.5,0) node[right] {$x$};
\draw[->] (3,-0.2)--(3,1.1) node[right] {$y$};

\draw (3,0)--(3+1.2,0);
\draw (3+1.2,0)--(3+0.707,0.707);
\draw (3,0)--(3+0.833*0.707,0.833*0.707);
\draw [thick,dotted](3+0.833*0.707,0.833*0.707)--(3+0.833*1.2,0);
\draw (3+0.707,0.707)--(3,0);

\fill (3+0.833*1.2,0) circle (0.02 ) node[below ] {\small$\tilde{B}_t$};
\fill (3+0.833*0.707,0.833*0.707) circle (0.02 ) node[above left] {\small$C_t$};
\fill (3, 0) circle (0.02 ) node[below ] {\small$A$};
\fill (4.2, 0) circle (0.02 ) node[below ] {\small$B_t$};
\fill (3+0.707, 0.707) circle (0.02 ) node[above right] {\small$C$};

\node at (3.5,0.2) {\small $\Omega_t$};

\end{tikzpicture}
\caption{An illustration of Theorem \ref{yangsheng1'}. Fix one angle, stretching one leg while scaling to maintain the area as fixed, then $D(\cdot)$ is strictly decreasing.}
\label{fig:ctm}
\end{figure}
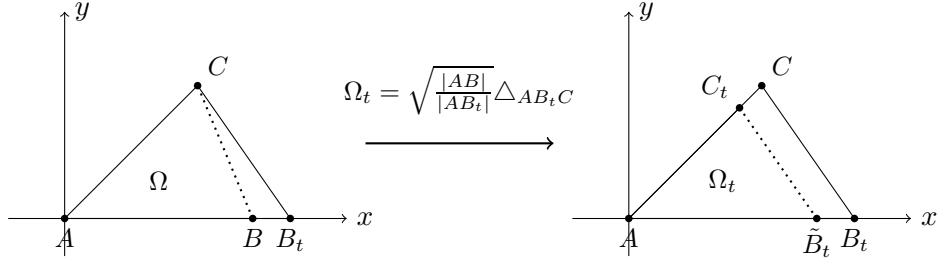

\begin{proof}[Proof of Theorem \ref{yangsheng1'}]
    Let $\alpha=\angle A$ and $\beta_t=\angle AB_tC=\angle A\tilde{B}_tC_t$. Without loss of generality, we assume that $|AB|=|AC|=1$.
Let 
    \begin{align*}
    G_t(x,y)=\left((1+t)x-t(\cot \alpha)y,y \right).
    \end{align*}
Such $G_t$ is constructed to be a linear mapping which maps $B$ to $B_t$ while keeping $A$ and $C$ fixed. Therefore, $G_t$ maps $\Omega=\triangle_{ABC}$ to the triangle $\triangle_{AB_tC}$, and hence
\begin{align*}
    F_t(x,y):=\frac{1}{\sqrt{1+t}}G_t(x,y)
\end{align*}
maps $\Omega$ to $\Omega_t$.

Let 
\begin{align*}
    \eta(t,x,y)=\frac{1}{2(1+t)}\left(x-2(\cot\alpha)y,-y\right).
\end{align*}
Such $\eta$ is obtained by direct computation to guarantee that $\tfrac{\partial}{\partial t}F_t(x,y)=\eta(t,F_t(x,y))$.

Recall that by Theorem \ref{Tshapederivative}, we have
\begin{align*}
    \frac{d}{dt}D(\Omega_t)=\int_{\partial \Omega_t}2V_{\Omega_t}((x,y)) \eta(t,x,y)\cdot  \nu\, ds.
\end{align*}
Note that on $AC_t$, $\nu=(-\sin\alpha, \cos\alpha)$, and thus
\begin{align*}
    \eta\cdot \nu=\frac{1}{2(1+t)}(-x\sin\alpha+(\cos\alpha)y)=0.
\end{align*}
On $A\tilde{B}_t$, $\nu=(0,-1)$, and thus
\begin{align*}
    \eta\cdot \nu=y=0.
\end{align*}
On $\tilde{B}_tC_t$, 
\begin{align*}
    \eta \cdot \nu=&\frac{1}{2(1+t)}\left(x-2\frac{\cos\alpha}{\sin \alpha}y,-y\right)\cdot (\sin \beta_t,\cos\beta_t)\\
    =&\frac{\sin\beta_t}{2(1+t)}\left(x-\left(2\frac{\cos\alpha}{\sin\alpha}+\frac{\cos \beta_t}{\sin \beta_t}\right)y\right).
\end{align*}
Let $M_t$ be the midpoint of $B_tC$ and $\theta_t=\angle M_tAB_t \in (0,\pi/2)$. 
By the law of sine,
\begin{align*}
    |AB_t|=\frac{\sin(\alpha+\beta_t)}{\sin \beta_t}, \quad |B_tM_t|=\frac{1}{2}|B_tC|=\frac{\sin \alpha}{2\sin\beta_t}.
\end{align*}
Again by the law of sine,
\begin{align*}
    \frac{\sin(\theta_t+\beta_t)}{|AB_t|}=\frac{\sin\theta_t}{|B_tM_t|} \quad \Rightarrow \quad   \cot \theta_t=2\frac{\cos\alpha}{\sin \alpha}+\frac{\cos\beta_t}{\sin \beta_t}.
\end{align*}
Hence, on $\tilde{B}_tC_t$,
\begin{align*}
    \eta\cdot \nu=\frac{\sin\beta_t}{2(1+t)}(x-(\cot\theta_t) y).
\end{align*}
Therefore, 
\begin{align*}
    \frac{d}{dt}D(\Omega_t)=\frac{\sin\beta_t}{\tan\theta_t}\frac{1}{1+t}\int_{\tilde{B}_tC_t}V_{\Omega_t}((x,y))((\tan \theta_t)x-y)\, ds.
\end{align*}
Now let $\tilde{M}_t$ be the intersection point of the line segment $AM_t$ and the line segment $\tilde{B}_tC_t$, and thus $\tilde{M}_t$ is exactly the midpoint of $\tilde{B}_tC_t$.

Note that for any $p_t=(x,y) \in \tilde{B}_t\tilde{M}_t$, we let $p_t'=(x',y')\in \tilde{M}_tC_t$ such that $|p_t'\tilde{M}_t|=|p_t\tilde{M}_t|$. When $t>0$, $|A\tilde{B}_t|>|AC_t|$, and thus by symmetry and Theorem \ref{symmetrycomparison},
we have
\begin{align*}
    0<x\tan\theta_t-y=-(x'\tan\theta_t-y'),\quad V_{\Omega_t}(p_t)<V_{\Omega_t} (p_t').
\end{align*}
Therefore, for $t>0$, we have
\begin{align*}
    \frac{d}{dt}D(\Omega_t)=\frac{\sin\beta_t}{\tan\theta_t}\frac{1}{(1+t)}\int_{\tilde{B}_t\tilde{M}_t}\left(V_{\Omega_t}(p_t)-V_{\Omega_t}(p_t')\right)\left(x\tan\theta-y\right)\, ds<0.
\end{align*}
This finishes the proof. 
\end{proof}

\begin{figure}[htp]
\centering
\begin{tikzpicture}[scale = 1.5]
\pgfmathsetmacro\d{180/pi};
\pgfmathsetmacro\S{sin(60)};
\pgfmathsetmacro\C{cos(60)};
\pgfmathsetmacro\m{\C+1.3};
\pgfmathsetmacro\mx{\C+1.3};
\pgfmathsetmacro\my{\S};
\pgfmathsetmacro\re{\my*(\S/(\C-1.3))+\mx};
\pgfmathsetmacro\px{1.2*\m};
\pgfmathsetmacro\pxx{0.8*\m};
\pgfmathsetmacro\py{\S/(\C-1.3)*(\px-2.6)};
\pgfmathsetmacro\pyy{\S/(\C-1.3)*(\pxx-2.6)};

\draw[->] (-0.5,0)--(3,0) node[below right] {$x$};
\draw[->] (0,-0.5)--(0,2.3) node[above left] {$y$};

\fill[green,opacity=0.4] (\re,0)--(2.6,0)--(\mx,\my)--cycle;
\fill[yellow,opacity=0.4] (\re,0)--(2*\C,2*\S)--(\mx,\my)--cycle;

\draw (0,0)--(2*\C,2*\S);
\draw (2*\C,2*\S)--(2.6,0);
\fill (0,0) circle (0.015) node[below left] {\small $A$};
\fill (2*\C,2*\S) circle (0.015) node[above right] {\small $C_t$};
\fill (2.6,0) circle (0.015) node[below right] {\small $\tilde{B}_{t}$};

\draw[samples=500,black, domain=0:pi/3, variable=\t] plot ({0.2*cos(\t*\d)},{0.2*sin(\t*\d)});
\node at (0.25,0.2) { $\alpha$};

\draw[samples=500,black, domain=3*pi/4:pi, variable=\t] plot ({2.6+0.2*cos(\t*\d)},{0.2*sin(\t*\d)});
\node at (2.3,0.15) {\small $\beta_{t}$};

\draw[samples=500,thick,dotted, domain=-0.5:2.5, variable=\t] plot ({\t},{\S/(1.3+\C)*\t});

\fill (1.3+\C,\S) circle (0.02) node[above] {\small $\tilde{M}_{t}$};

\draw[samples=500,black, domain=0:pi/7, variable=\t] plot ({0.5*cos(\t*\d)},{0.5*sin(\t*\d)});
\node at (0.7,0.15) {\small  $\theta_{t}$};

\fill (\pxx,\pyy) circle (0.015) node[above right] {\small $p'_{t}$};
\fill (\px,\py) circle (0.015) node[above right] {\small $p_{t}$};

\end{tikzpicture}
\caption{An illustration of the proof of Theorem \ref{yangsheng1'}. $\Omega_t=\triangle_{A\tilde{B}_tC_t}$. For any $t>0$, $|AB_t|>|AC|$. $\tilde{M}_t$ is the midpoint of 
$\tilde{B}_tC_t$, $\theta_t=\angle \tilde{B}_tA\tilde{M}_t$. If $|p_t \tilde{M}_t|=|p_t'\tilde{M}_t|$, then $V_{\Omega_t}(p_t)<V_{\Omega_t}(p_t')$.}
\label{fig:add2}
\end{figure}
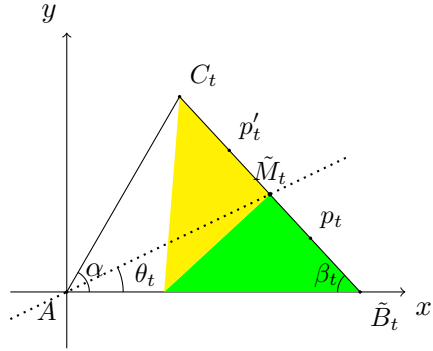

Next, we will prove Theorem \ref{zhenxi1}, and it suffices to prove the following theorem. 

\begin{theorem}
\label{zhenxi1'}
	Let $\Omega=\triangle _{ABC}$ be the quadrilateral with vertices $A=(-b,h)$, $B=(a,0)$, $C=(-a,0)$, where $a,b,h>0$.
	Let $A_t=((t-1)b,h)$ and $\Omega _t=\triangle _{A_tBC}$.
	Then, $D( \Omega_t)$ is strictly increasing when $-\infty< t\le 1$.
\end{theorem}

The argument in the proof of Theorem \ref{yangsheng1'} no longer works. Instead, slightly adapting the argument in \cite{BCT} can conclude Theorem \ref{zhenxi1'}. The original idea is from \cite{CHVY}.

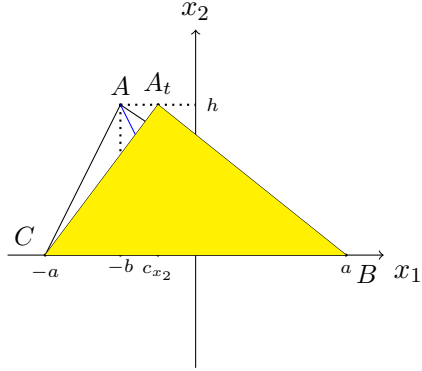
\begin{figure}[htp]
\centering
\begin{tikzpicture}
\draw (-2,0)--(-1,2)--(2,0)--(-2,0);
\draw[->] (-2.5,0)--(2.5,0) node[below right] {$x_{1}$};
\draw[->] (0,-1.5)--(0,3) node[above] {$x_{2}$};
\draw[dotted, thick] (-1,2)--(0,2) node[right] {\tiny $h$};
\draw[dotted, thick] (-1,2)--(-1,0);
\fill (-1,0) circle (0.02);
\node at (-1,-0.15) {\tiny $-b$};
\draw (-2,0)--(-0.5,2)--(2,0);
\fill (-2,0) circle (0.02) node[above left] {\small $C$};
\fill (-1,2) circle (0.02) node[above] {\small $A$};
\fill (-0.5,2) circle (0.02) node[above] {\small $A_{t}$};
\fill (2,0) circle(0.02) node[below right] {\small $B$};
\draw[blue] (-1,2)--(0,0);
\draw[thick,blue] (-0.5,1)--(-0.5,0);
\fill (-0.5,0) circle (0.02) node[below] {\tiny $c_{x_{2}}$};
\draw[dotted, thick] (-0.5,1)--(0,1) ;
\node[below] at (-2,0) {\tiny $-a$};
\node[below] at (2,0) {\tiny $a$};
\fill[yellow, opacity=0.3] (-2,0)--(-0.5,2)--(2,0)--(-2,0);
\end{tikzpicture}
\label{fig1}
\caption{Parallel movement of top vertex}
\end{figure}

\begin{proof}[Proof of Theorem \ref{zhenxi1'}]
    Let
    $$
	\varPhi _t\left( x_1,x_2 \right):=
		\left( x_1+\frac{x_2}{h}bt,x_2 \right).
	$$
Then $\varPhi_t$ is a smooth mapping that transforms $\Omega$ into $\Omega_t$, and $\Omega_1$ becomes an isosceles triangle with $|A_1C|=|A_1B|$.

For any $x_2\in \mathbb{R}$, we let
$$P_{x_2}=\{x_1\in \mathbb{R}: (x_1,x_2)\in \Omega\}.$$
Hence
\begin{align*}
    P_{x_2}=(c_{x_2}-r_{x_2},c_{x_2}+r_{x_2}),
\end{align*}
where $c_{x_2}=-x_2b/h$ and $r_{x_2}=a(h-x_2)/h$.

For any fixed $x_2,y_2\in (0,h)$, we define $K_l(r)=K(\sqrt{l^2+r^2})$, where $l=|x_2-y_2|$. Hence
\begin{align*}
    D(\Omega_t)=&\int_{\Omega_t}\left(\int_{\Omega_t} K\left(\sqrt{(x_1-y_1)^2+(x_2-y_2)^2}\right)\,dx_1dx_2\right)\, dy_1dy_2\\
    =& \int_0^h\int_0^hI_{K_l}[P_{x_2},P_{y_2}](t)\, dx_2dy_2,
\end{align*}
where
\begin{align*}
    I_{K_l}[P_{x_2},P_{y_2}](t)=\int_{\varPhi_t(P_{y_2})}\int_{\varPhi_t(P_{x_2})}K_l(x_1-y_1)dx_1dy_1.
\end{align*}
Since
\begin{align*}
    c_{x_2}-c_{y_2}=\frac{b}{h}(y_2-x_2),
\end{align*}
we can rewrite
\begin{align*}
    I_{K_l}[P_{x_2},P_{y_2}](t)=&\int_{c_{y_2}-r_{y_2}+\frac{b}{h}y_2t}^{c_{y_2}+r_{y_2}+\frac{b}{h}y_2t}\int_{c_{y_2}-r_{y_2}+\frac{b}{h}x_2t}^{c_{x_2}+r_{x_2}+\frac{b}{h}x_2t}K_l(x_1-y_1)\, dx_1dy_2\\
    =&\int_{-r_{y_2}}^{r_{y_2}}\int_{-r_{x_2}}^{r_{x_2}}K_l\left(x_1-y_1+c_{x_2}-c_{y_2}+\left(\frac{b}{h}x_2-\frac{b}{h}y_2\right)t\right)\,dx_1dy_1\\
    =&\int_{-r_{y_2}}^{r_{y_2}}\int_{-r_{x_2}}^{r_{x_2}}K_l\left(x_1-y_1+\frac{b}{h}(x_2-y_2)(t-1)\right)\,dx_1dy_1.
\end{align*}
Let $$M_t:=\frac{b}{h}(x_2-y_2)(t-1),$$and hence
\begin{align*}
    \frac{d}{dt}I_{K_l}[P_{x_2},P_{y_2}](t)=&\frac{b}{h}(x_2-y_2)\int_{-r_{y_2}}^{r_{y_2}}\int_{-r_{x_2}}^{r_{x_2}}K_l'\left(x_1-y_1+M_t\right)\,dx_1dy_1\\
    =&\frac{b}{h}(x_2-y_2)\int_{-r_{y_2}}^{r_{y_2}}\int_{-r_{x_2}+M_t}^{r_{x_2}+M_t}K_l'\left(x_1-y_1\right)\,dx_1dy_1.
\end{align*}
If $x_2<y_2$, then for $t<1$, $M_t>0$. Let $R_t=(-r_{x_2}+M_t,r_{x_2}+M_t)\times (-r_{y_2},r_{y_2})$, $E=R_t \cap \{(x_1,y_1): y_1>x_1\}$ and $E'$ be the reflection of $E$ about the origin. Let $\Sigma=R_t\setminus (E\cup E')$. Then in this case, $\Sigma$ is a non-empty rectangle lying inside $\{(x_1,y_1): x_1-y_1>0\}$, as illustrated in Figure \ref{figrecpicture}. Since $K_l'(r)$ is an odd function and $K_l'(r)<0$ if $r>0$, we have 
\begin{align*}
    \frac{d}{dt}I_{K_l}[P_{x_2},P_{y_2}](t)=\frac{b}{h}(x_2-y_2)\iint_{\Sigma}K_l'(x_1-y_1)\, dx_1dy_1>0.
\end{align*}
\begin{figure}[htp]
\centering
\begin{tikzpicture}
\draw[->] (-3,0)--(3,0) node[below right] {$x_{1}$};
\draw[->] (0,-2)--(0,2) node[above left] {$y_{1}$};
\draw[dashed] (-2,1)--(2,1)--(2,-1)--(-2,-1)--(-2,1);
\node[above right] at (0,1) {\small $r_{y_{2}}$};
\node[below right] at (0,-1) {\small $-r_{y_{2}}$};
\draw (-1.5,1)--(2.5,1)--(2.5,-1)--(-1.5,-1)--(-1.5,1);
\fill (-1.5,1) circle(0.03);\fill (-1.5,-1) circle(0.03);
\fill (-1.5,0) circle(0.03);\fill (2.5,0) circle(0.03);\fill (2.5,-1) circle(0.03);\fill (2.5,1) circle(0.03);\fill (1.5,-1) circle(0.03);\fill (1.5,1) circle(0.03);
\node[below] at (-1.5,0) {\tiny $-r_{x_{2}}+M_{t}$};
\node[below] at (2.5,0) {\tiny $r_{x_{2}}+M_{t}$};
\node[above] at (-1.5,1) {\tiny $P_{1}$};\node[above] at (2.5,1) {\tiny $P_{2}$};\node[below] at (2.5,-1) {\tiny $P_{3}$};\node[below] at (-1.5,-1) {\tiny $P_{4}$};
\node[above right] at (1.5,1) {\tiny $P_{1}^{'}$};\node[below] at (1.5,-1) {\tiny $P_{4}^{'}$};
\draw[smooth,domain=-1.7:1.7, variable=\t] plot 
({\t}, {\t}) node[above] {\small $y_{1}=x_{1}$};
\draw[dashed] (1.5,1)--(1.5,-1);
\fill[gray, opacity=0.2] plot (1.5,1)--(1.5,-1)--(2.5,-1)--(2.5,1)--(1.5,1);
\node at (2,0.7) {\small $\Sigma$};
\end{tikzpicture}
\caption{The case of $x_2<y_2$ and thus $M_t>0$.}
\label{figrecpicture}
\end{figure}
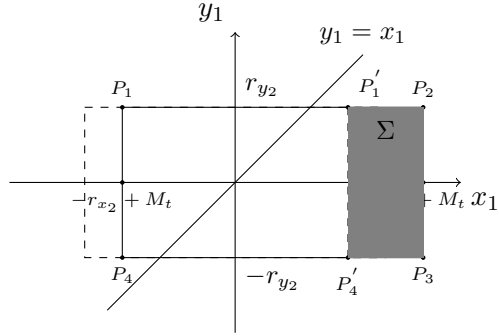

Similarly, if $x_2>y_2$, we still have $\tfrac{d}{dt}I_{K_l}[P_{x_2},P_{y_2}](t)>0$, for $t<1$.

Therefore, for any $-\infty<t<1$, we have
\begin{align*}
    \frac{d}{dt}D(\Omega_t)=\int_0^h\int_0^h \frac{d}{dt}I_{K_l}[P_{x_2},P_{y_2}](t)\, dx_2 dy_2>0. 
\end{align*}
\end{proof}

\section{Monotonicity results of $D(\cdot)$ on rhombuses and rectangles via vertical stretching flows}

In this section, we prove Theorem \ref{rhombustheoremc} and Theorem \ref{recderipos}. The proof is based on the same stretching flow argument as in the proof of Theorem \ref{yangsheng2}.

First, to prove Theorem \ref{rhombustheoremc}, it suffices to prove the following, which gives the monotonicity of $D(\cdot)$ along the continuous diagonal-stretching deformation on a rhombus.

\begin{theorem}
\label{equirh}
    Let $R_t, \, t\ge 0$ be a family of rhombuses with vertices $A=(-a,0)$, $B_t=(0,-a(1+t))$, $C=(a,0)$, $D_t=(0,a(1+t))$, where $a>0$. Let $\Omega_t=\tfrac{1}{\sqrt{1+t}}R_t$. Then,
    \begin{align*}
        \frac{d}{dt}D(\Omega_t)\le 0,
    \end{align*}
   where the equality holds if and only if $t=0$.
\end{theorem}

\begin{proof}
    Without loss of generality, we assume that $a=1$. 
Again, we let $$F_t(x,y)=\frac{1}{\sqrt{1+t}}\left(x, (1+t)y\right), \quad  \eta(t,x,y)=\frac{1}{2}\frac{1}{1+t}(-x,y).$$
Then $F_t(R_0)=\Omega_t$ and $\eta$ is the smooth vector field generating the flow map $F_t$.

Let $A_t=\tfrac{1}{\sqrt{1+t}}A$, $\tilde{B}_t=\tfrac{1}{\sqrt{1+t}}B_t$, $C_t=\tfrac{1}{\sqrt{1+t}}C$, and $\tilde{D}_t=\tfrac{1}{\sqrt{1+t}}D_t$.
By Theorem \ref{Tshapederivative} and the symmetry, we have
\begin{align*}
    \frac{d}{dt}D(\Omega_t)=\frac{4}{1+t}\int_{C_t\tilde{D}_t}V_{\Omega_t}((x,y))(-x,y)\cdot \nu\, ds,
\end{align*}

Note that on $C_t\tilde{D}_t$, 
\begin{align*}
    (-x,y)\cdot \nu=(-x,y)\cdot(\sin\theta_t,\cos\theta_t)=\cos\theta_t(y-(\tan\theta_t)x),
\end{align*}
where $\theta_t=\angle ACD_t$.

Hence
\begin{align*}
    \frac{d}{dt}D(\Omega_t)=\frac{4\cos\theta_t}{1+t}\int_{C_t\tilde{D}_t}V_{\Omega_t}((x,y))\left(y-(\tan\theta_t)x\right)\, ds.
\end{align*}

Let $M_t$ be the midpoint of $C_t\tilde{D}_t$. For any $p_t=(x,y)\in C_tM_t$ and $p'_t=(x',y')\in M_t\tilde{D}_t$ with $|p_tM_t|=|p_t'M_t|$, we have by the geometry of the 
rhombus and the reflection method (similar to the proof of Theorem \ref{symmetrycomparison}) that
\begin{align*}
    V_{\Omega_t}(p_t)\ge V_{\Omega_t}(p'_t),
\end{align*}
with $"="$ holding if and only if $t=0$.

Also, $-(y'-x'\tan \theta_t)=y-x\tan \theta_t<0$. Hence
\begin{align*}
    \frac{d}{dt}D(\Omega_t)=\frac{4\cos \theta_t}{1+t}\int_{C_tM_t}\left(V_{\Omega_t}(p_t)-V_{\Omega_t}(p_t')\right)(y-x\tan\theta_t)\, ds\le 0,
\end{align*}
where the equality holds if and only if $t=0$. This finishes the proof.

\end{proof}

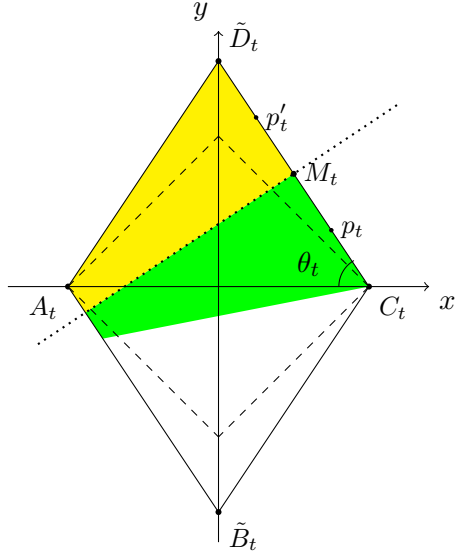
\begin{figure}[htp]
\centering
\begin{tikzpicture}[scale = 2]
\pgfmathsetmacro\a{1};
\pgfmathsetmacro\t{0.5};
\pgfmathsetmacro\d{180/pi};
\pgfmathsetmacro\at{(1+\t)*\a};
\pgfmathsetmacro\xt{1/(\a/\at+\at/\a)*(\a*\a/(2*\at)-\at-\at/2)};
\pgfmathsetmacro\yt{-\at/\a*(\xt+\a)};
\pgfmathsetmacro\rx{1/(\a/(2*\at)+\at/(2*\a))*(-\at+\a*\a/\at)};
\pgfmathsetmacro\ry{-\at/\a*(\rx+\a)};

\fill (-\a,0) circle(0.02) node[below left] {\small $A_t$};
\fill (\a,0) circle(0.02) node[below right] {\small $C_t$};

\fill[yellow,opacity=0.4] (\xt,\yt)--(-\a,0)--(0,\at)--(\a/2,\at/2)--cycle;
\fill[green,opacity=0.4] (\xt,\yt)--(\rx,\ry)--(\a,0)--(\a/2,\at/2)--cycle;
\draw (-\a,0)--(0, \at);
\draw (\a,0)--(0, \at);
\draw (-\a,0)--(0,-\at);
\draw (\a,0)--(0,-\at);

\draw[->] (-1.4*\a,0)--(1.4*\a,0) node[below right] {$x$};
\draw[->] (0,-1.7*\a)--(0,1.7*\a) node[above left] {$y$};

\draw[dashed] (-\a,0)--(0, \a);
\draw[dashed] (\a,0)--(0, \a);
\draw[dashed] (-\a,0)--(0,-\a);
\draw[dashed] (\a,0)--(0,-\a);

\fill (0,-\at) circle(0.02) node[below right] {\small $\tilde{B}_{t}$};
\fill (0,\at) circle(0.02) node[above right] {\small $\tilde{D}_{t}$};
\fill (\a/2,\at/2) circle(0.02) node[right] {\small $M_{t}$};

\fill (\a/4,3*\at/4) circle(0.015) node[right] {\small $p'_{t}$};
\fill (3*\a/4,\at/4) circle(0.015) node[right] {\small $p_{t}$};

\draw[samples=500,thick,black,dotted, domain=-1.2*\a:1.2*\a, variable=\s] plot ({\s},{\a/\at*(\s-\a/2)+\at/2});

\draw[samples=500,black, domain=2*pi/3:pi, variable=\t] plot ({\a+0.2*cos(\t*\d)},{0.2*sin(\t*\d)});
\node at (0.6*\a,0.15*\a) {$\theta_{t}$};

\end{tikzpicture}
\caption{An illustration of the proof of Theorem \ref{equirh}. The reflection of the yellow part is strictly inside the rhombus $\Omega_t$.}
\label{fig:ctmrhom}
\end{figure}

Next, we prove Theorem \ref{recderipos}. Clearly, to prove it, it suffices to prove the following theorem.

\begin{theorem}
\label{recderipos'}
    Let $R=\square ABCD$ be the square with vertices $A=(0,0)$, $B=(b,0)$, $C=(b,b)$ and $D=(0,b)$, $b>0$. Let $C_t=(b,(1+t)b)$, $D_t=(0,(1+t)b)$ and $R(t)=\square ABC_tD_t$ be the rectangle with vertices $A,B,C_t,D_t$. Let $\Omega_t=\tfrac{1}{\sqrt{1+t}}R(t)$. Then, for any $t\ge 0$,
    \begin{align*}
        \frac{d}{dt}D(\Omega_t)\le  0,
    \end{align*}
where the equality holds if and only if $t=0$.
\end{theorem}

\begin{proof}
Without loss of generality, we assume $b=1$. As before, we let
    $$F_t(x,y)=\frac{1}{\sqrt{1+t}}\left(x, (1+t)y\right), \quad  \eta(t,x,y)=\frac{1}{2}\frac{1}{1+t}(-x,y).$$
Then $F_t(R)=\Omega_t$ and $\eta$ is the smooth vector field generating the flow map $F_t$.

Let $B_t=\tfrac{1}{\sqrt{1+t}}B$, $\tilde{C}_t=\tfrac{1}{\sqrt{1+t}}C_t$, and $\tilde{D}_t=\tfrac{1}{\sqrt{1+t}}D_t$. Then $\Omega=\square AB_t\tilde{C}_t\tilde{D}_t$ is the rectangle with vertices $A,B_t,\tilde{C}_t,\tilde{D}_t$. Clearly, $|B_t\tilde{C}_t|=\tfrac{1}{\sqrt{1+t}}|BC|=\sqrt{1+t}$, and $|\tilde{C}_t\tilde{D}_t|=\tfrac{1}{\sqrt{1+t}}$.

Since $\eta\cdot \nu=0$ on $AB_t\cup A\tilde{D}_t$, by Theorem \ref{Tshapederivative}, we have
\begin{align*}
    \frac{d}{dt}D(\Omega_t)=&\frac{1}{1+t}\int_{B_t\tilde{C}_t}V_{\Omega_t}((x,y))(-x,y)\cdot (1,0)\, ds+\frac{1}{1+t}\int_{\tilde{C}_t\tilde{D}_t}V_{\Omega_t}((x,y))(-x,y)\cdot (0,1)\, ds\\
    =&\frac{1}{1+t}\left(-\frac{1}{\sqrt{1+t}}\int_{B_t\tilde{C}_t}V_{\Omega_t}((x,y))\, ds+\sqrt{1+t}\int_{\tilde{C}_t\tilde{D}_t}V_{\Omega_t}((x,y))\, ds\right)\\
    =&\frac{1}{1+t}\left(\frac{1}{|\tilde{C}_t\tilde{D}_t|}\int_{\tilde{C}_t\tilde{D}_t}V_{\Omega_t}((x,y))\, ds-\frac{1}{|B_t\tilde{C}_t|}\int_{B_t\tilde{C}_t}V_{\Omega_t}((x,y))\, ds\right)
\end{align*}
When $t=0$, $|\tilde{C}_t\tilde{D}_t|=|B_t\tilde{C}_t|=1$, and hence $\tfrac{d}{dt}D(\Omega_t)=0$. When $t>0$, $|\tilde{C}_t\tilde{D}_t|<|B_t\tilde{C}_t|$, and hence by Proposition \ref{not-equidistribution-thmintro} and above, $\tfrac{d}{dt}D(\Omega_t)<0$. This finishes the proof. 
\end{proof}

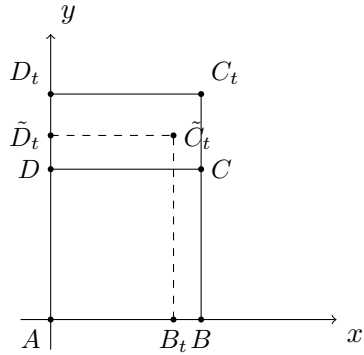
\begin{figure}[htp]
\centering
\begin{tikzpicture}[scale = 2]
\pgfmathsetmacro\xlength{1.5};
\pgfmathsetmacro\ylength{2};
\pgfmathsetmacro\Mx{1/2*\xlength};
\pgfmathsetmacro\My{1/2*\ylength};
\pgfmathsetmacro\t{0.5};

\fill (0, 0) circle (0.02) node[below left] {\small ${A}$};
\fill (1, 0) circle (0.02) node[below] {\small ${B}$};
\fill (1, 1) circle (0.02) node[right] {\small $C$};
\fill (1, 1.5) circle (0.02) node[above right] {\small ${C}_t$};
\fill (0, 1) circle (0.02) node[left] {\small ${D}$};
\fill (0, 1.5) circle (0.02) node[above left] {\small ${D}_t$};
\fill (0, 0.8165*1.5) circle (0.02) node[left] {\small $\tilde{D}_t$};
\fill (0.8165,0) circle (0.02) node[below] {\small $B_t$};
\fill (0.8165, 0.8165*1.5) circle (0.02) node[right] {\small $\tilde{C}_t$};

\draw[->] (-0.2,0)--(1.9,0)node[below right] {$x$};
\draw (0,1)--(1,1);
\draw (1,0)--(1,1.5);
\draw (1,1.5)--(0,1.5);
\draw[->] (0,-0.2)--(0,1.9)node[above right] {$y$};
\draw [dashed] (0,0.8165*1.5)--(0.8165*1,0.8165*1.5);
\draw [dashed] (0.8165,0)--(0.8165,0.8165*1.5);

\end{tikzpicture}
\caption{An illustration of Theorem \ref{recderipos'}. $|\square AB_t\tilde{C}_t\tilde{D}_t|=|\square ABCD|$}
\label{fig-rectangle-1}
\end{figure}

\section{Continuous monotonic deformation of quadrilaterals}

In this section, we first prove Theorem \ref{quadritheorem}. The proof will be divided into two cases. 

Case 1: $A_0$ and $C_0$ lie on the same side of the $x_2$-axis. That is, $x_Ax_C>0$, where $x_A$ and $x_C$ are the first coordinates of $A_0$ and $C_0$, respectively. Then it suffices to prove the following:

\begin{proposition}
    \label{T2}
  	Let $\Omega=\Box _{ABCD}$ be the square with vertices $A=(-b_1,h_1)$, $B=(a,0)$, $C=(-b_2,-h_2)$, $D=(-a,0)$, where $a,b_1,b_2,h_1,h_2>0$.
	Let $A_t=(-b_1(1-t),h_1)$ and $C_t=(-b_2(1-t),-h_2)$, $\Omega _t=\Box _{A_tBC_tD}$. Then,
	$D\left( \Omega _t \right)$ is strictly increasing when $-\infty<t\leqslant 1$.  
\end{proposition}

\begin{figure}[htp]
\centering
\begin{tikzpicture}
\draw (-2,0)--(-1,2)--(2,0)--(-1.5,-1)--(-2,0);
\draw[->] (-2.5,0)--(2.5,0) node[below right] {$x_{1}$};
\draw[->] (0,-1.5)--(0,3) node[above] {$x_{2}$};
\draw[dotted, thick] (-1,2)--(0,2) node[right] {\tiny $h_{1}$};
\draw[dotted, thick] (-1,2)--(-1,0);
\fill (-1,0) circle (0.02);
\node at (-1,-0.15) {\tiny $-b_{1}$};
\draw (-2,0)--(-0.5,2)--(2,0);
\fill (-2,0) circle (0.02) node[below left] {\small $D$};
\fill (-1,2) circle (0.02) node[above] {\small $A$};
\fill (-0.5,2) circle (0.02) node[above] {\small $A_{t}$};
\fill (2,0) circle(0.02) node[below right] {\small $B$};
\fill (-1.5,-1) circle(0.02) node[below] {\small $C$};
\draw[dotted,thick ] (-1.5,-1)--(0,-1) node[right] {\tiny $-h_{2}$};
\draw[dotted,thick ] (-1.5,-1)--(-1.5,0) node[above] {\tiny $-b_{2}$};
\draw[blue] (-1,2)--(0,0);
\draw[dotted, thick,blue] (-0.5,1)--(-0.5,0);
\fill (-0.5,0) circle (0.02) node[below] {\tiny $C_{x_{2}}$};
\node[below] at (-2,0) {\tiny $-a$};
\node[below] at (2,0) {\tiny $a$};
\draw (2,0)--(-0.75,-1)--(-2,0);
\fill (-0.75,-1) circle (0.02) node[below] {$C_{t}$};
\fill[yellow, opacity=0.3] (-2,0)--(-0.5,2)--(2,0)--(-0.75,-1)--(-2,0);
\end{tikzpicture}
\label{fig1}
\caption{$A$ and $C$ lie in the same side of the $x_2$-axis}
\end{figure}
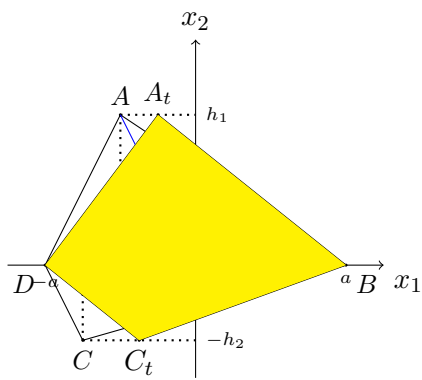

\vskip 0.2cm

\begin{proof}
    Let 
    $$
	\varPhi _t\left( x_1,x_2 \right) :=\begin{cases}
		\left( x_1+\frac{b_1}{h_1}x_2t,x_2 \right),\quad &\textit{$x_2\geqslant 0$}\\
		\left( x_1-\frac{b_2}{h_2}x_2t,x_2 \right),&\textit{$x_2<0$}\\
	\end{cases}
	\\$$
Then $\varPhi_t$ is a smooth mapping that transforms $\Omega$ into $\Omega_t$. As in the proof of Theorem \ref{zhenxi1'}, we let
$$P_{x_2}=\{x_1\in \mathbb{R}: (x_1,x_2)\in \Omega\}=\left( c_{x_2}-r_{x_2},c_{x_2}+r_{x_2} \right),$$
where in this case,
$$
	c_{x_2}=\begin{cases}
		-\frac{x_2}{h_1}b_1,\quad &\textit{$x_2\geqslant 0$}\\
		\frac{x_2}{h_2}b_2,&\textit{$x_2<0$}\\
	\end{cases}
	$$
and
$$
	r_{x_2}=\begin{cases}
		\frac{a}{h_1}\left( h_1-x_2 \right),\quad &\textit{$x_2\geqslant 0$}\\
		\frac{a}{h_2}\left( h_2+x_2 \right),&\textit{$x_2<0$}\\
	\end{cases}
	\\$$
For any fixed $x_2,y_2\in (0,h)$, we define $K_l(r)=K(\sqrt{l^2+r^2})$, where $l=|x_2-y_2|$. Hence
\begin{align*}
\frac{d}{dt}D(\Omega_t)=&\int_0^{h_1}\int_0^{h_1}\frac{d}{dt}I_{K_l}\left[ P_{x_2},P_{y_2} \right] \left( t \right)\, dx_2dy_2+2\int_{-h_2}^0 \int_0^{h_1}\frac{d}{dt}I_{K_l}\left[ P_{x_2},P_{y_2} \right] \left( t \right)\, dx_2dy_2\\
&+\int_{-h_2}^0\int_{-h_2}^0\frac{d}{dt}I_{K_l}\left[ P_{x_2},P_{y_2} \right] \left( t \right) \,dx_2dy_2\\
=:&I+II+III,
\end{align*}
where
\begin{align*}
    I_{K_l}[P_{x_2},P_{y_2}](t)=\int_{\varPhi_t(P_{y_2})}\int_{\varPhi_t(P_{x_2})}K_l(x_1-y_1)dx_1dy_1.
\end{align*}
As in the proof of Theorem \ref{zhenxi1'}, for $t<1$, we have $I>0$ and $III>0$.

For $x_2\in \left[ 0,h_1 \right]$, $y_2\in \left[ -h_2,0 \right) $, we have
	$$
	\begin{aligned}
		I_{K_l}\left[ P_{x_2},P_{y_2} \right] \left( t \right) &=\int_{c_{y_2}-r_{y_2}-\frac{b_2}{h_2}y_2t}^{c_{y_2}+r_{y_2}-\frac{b_2}{h_2}y_2t}\int_{c_{x_2}-r_{x_2}+\frac{b_1}{h_1}x_2t}^{c_{x_2}+r_{x_2}+\frac{b_1}{h_1}x_2t}K_l\left( x_1-y_1 \right) \,dx_1dy_1\\
		&=\int_{-r_{y_2}}^{r_{y_2}}\int_{-r_{x_2}}^{r_{x_2}}K_l\left( x_1-y_1+c_{x_2}-c_{y_2}+\frac{b_1}{h_1}x_2t+\frac{b_2}{h_2}y_2t \right)\, dx_1dy_1\\
	\end{aligned}
	$$
Let $$M_t:=c_{x_2}-c_{y_2}+\frac{b_1}{h_1}x_2t+\frac{b_2}{h_2}y_2t=\left( \frac{b_1}{h_1}x_2+\frac{b_2}{h_2}y_2 \right) \left( t-1 \right) $$
Then
\begin{align*}
    \frac{d}{dt}I_{K_l}\left[ P_{x_2},P_{y_2} \right] \left( t \right)=\left( \frac{b_1}{h_1}x_2+\frac{b_2}{h_2}y_2 \right)\int_{-r_{y_2}}^{r_{y_2}}\int_{-r_{x_2}+M_t}^{r_{x_2}+M_t}K_l'\left(x_1-y_1\right)\,dx_1dy_1.
\end{align*}
Since for $t<1$, $\tfrac{b_1}{h_1}x_2+\tfrac{b_2}{h_2}y_2$ and $M_t$ have opposite signs, similar to the proof of Theorem \ref{zhenxi1'}, we also have
$II>0$ for $t<1$.

Therefore, we conclude the proposition.
\end{proof}
\vskip 0.2cm

Case 2: $A_0$ and $C_0$ lie on two different sides of the $x_2$-axis. That is, $x_Ax_C\le 0$. Then it suffices to prove the following stronger result.

\begin{proposition}
    \label{T1}
	Let $\Omega=\Box _{ABCD}$ be the quadrilateral with vertices $A=(-b_1,h_1)$, $B=(a,0)$, $C=(b_2,-h_2)$, $D=(-a,0)$, where $a,b_1,h_1,h_2>0$, and $b_2\ge 0$.
	Let $A_t=(b_1(t-1),h_1)$ and $\Omega _t=\Box _{A_tBCD}$.
	Then, $D( \Omega _t)$ is strictly increasing when $-\infty<t\leqslant 1$.
\end{proposition}

\begin{figure}[htp]
\centering
\begin{tikzpicture}
\draw (-2,0)--(-1,2)--(2,0)--(1,-1)--(-2,0);
\draw[->] (-2.5,0)--(2.5,0) node[below right] {$x_{1}$};
\draw[->] (0,-1.5)--(0,3) node[above] {$x_{2}$};
\draw[dotted, thick] (-1,2)--(0,2) node[right] {\tiny $h_{1}$};
\draw[dotted, thick] (-1,2)--(-1,0);
\fill (-1,0) circle (0.02);
\node at (-1,-0.15) {\tiny $-b_{1}$};
\draw (-2,0)--(-0.5,2)--(2,0);
\fill (-2,0) circle (0.02) node[above left] {\small $D$};
\fill (-1,2) circle (0.02) node[above] {\small $A$};
\fill (-0.5,2) circle (0.02) node[above] {\small $A_{t}$};
\fill (2,0) circle(0.02) node[below right] {\small $B$};
\fill (1,-1) circle(0.02) node[below] {\small $C$};
\draw[dotted,thick ] (1,-1)--(0,-1) node[left] {\tiny $-h_{2}$};
\draw[dotted,thick ] (1,-1)--(1,0) node[above] {\tiny $b_{2}$};
\draw[blue] (-1,2)--(0,0);
\draw[thick,blue] (-0.5,1)--(-0.5,0);
\fill (-0.5,0) circle (0.02) node[below] {\tiny $C_{x_{2}}$};
\node[below] at (-2,0) {\tiny $-a$};
\node[below] at (2,0) {\tiny $a$};
\fill[yellow, opacity=0.3] (-2,0)--(-0.5,2)--(2,0)--(1,-1)--(-2,0);
\end{tikzpicture}
\label{fig1}
\caption{$A$ and $C$ lie on two different sides of the $x_2$-axis}
\end{figure}
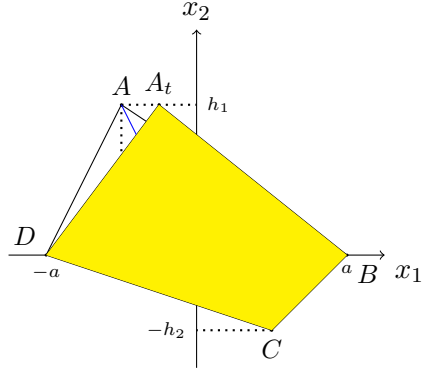

\begin{proof}
    Let
  $$
	\varPhi _t\left( x_1,x_2 \right):=
	\begin{cases}
		\left( x_1+\frac{x_2}{h_1}b_1t,x_2 \right), \quad &\textit{$x_2\geqslant 0$}\\
		\left( x_1,x_2 \right),&\textit{$x_2<0$}
	\end{cases},
	$$ and thus $\Omega_t=\varPhi_t(\Omega)$. Again,
    \begin{align*}
\frac{d}{dt}D(\Omega_t)=&\int_0^{h_1}\int_0^{h_1}\frac{d}{dt}I_{K_l}\left[ P_{x_2},P_{y_2} \right] \left( t \right)\, dx_2dy_2+2\int_{-h_2}^0 \int_0^{h_1}\frac{d}{dt}I_{K_l}\left[ P_{x_2},P_{y_2} \right] \left( t \right)\, dx_2dy_2\\
&+\int_{-h_2}^0\int_{-h_2}^0\frac{d}{dt}I_{k_l}\left[ P_{x_2},P_{y_2} \right] \left( t \right) \,dx_2dy_2\\
=:&I+II+III,
\end{align*}
where
$$
	c_{x_2}=\begin{cases}
		-\frac{x_2}{h_1}b_1, \quad &\textit{$x_2\geqslant 0$}\\
		-\frac{x_2}{h_2}b_2,&\textit{$x_2<0$}\\
	\end{cases}
	$$

	$$
	r_{x_2}=\begin{cases}
		\frac{a}{h_1}\left( h_1-x_2 \right), \quad &\textit{$x_2\geqslant 0$}\\
		\frac{a}{h_2}\left( h_2+x_2 \right),&\textit{$x_2<0$}\\
	\end{cases}
	$$
	and
	$$
	P_{x_2}:=\left\{ x_1\in \mathbb{R} , \left( x_1,x_2 \right) \in \Omega \right\} =\left( c_{x_2}-r_{x_2},c_{x_2}+r_{x_2} \right).
	$$
Similar to the proof of Theorem \ref{zhenxi1'}, $I>0$. Also, it is clear that $III=0$.

If $0<x_2<h_1$ and $-h_2<y_2<0$, then 
\begin{align*}
    I_{K_l}[P_{x_2},P_{y_2}](t)=&\int_{c_{y_2}-r_{y_2}}^{c_{y_2}+r_{y_2}}\int_{c_{x_2}-r_{x_2}+\frac{b_1}{h_1}x_2t}^{c_{x_2}+r_{x_2}+\frac{b_1}{h_1}x_2t} K_l(x_1-y_1)\, dx_1dy_1\\
    =& \int_{-r_{y_2}}^{r_{y_2}}\int_{-r_{x_2}}^{r_{x_2}}K_l\left(x_1-y_1+c_{x_2}+\frac{b_1}{h_1}x_2t-c_{y_2}\right)\, dx_1dy_1\\
    =&\int_{-r_{y_2}}^{r_{y_2}}\int_{-r_{x_2}}^{r_{x_2}}K_l\left(x_1-y_1+M_t\right)\, dx_1dy_1,
\end{align*}
where
\begin{align}
\label{bx}
    M_t=-\frac{x_2}{h_1}b_1+\frac{y_2}{h_2}b_2+\frac{b_1}{h_1}x_2t=\frac{b_1}{h_1}x_2(t-1)+\frac{y_2}{h_2}b_2.
\end{align}
Hence
\begin{align*}
    II=2\int_{-h_2}^0\int_0^{h_1} \frac{b_1x_2}{h_1}\left(\int_{-r_{y_2}}^{r_{y_2}}\int_{-r_{x_2}+M_t}^{r_{x_2}+M_t}K_l'(x_1-y_1)\, dx_1 dy_1\right)\, dx_2dy_2.
\end{align*}
From \eqref{bx}, when $t<1$, $x_2>0$ and $y_2<0$, we have $M_t<0$. Then, the similar argument in the proof of Theorem \ref{zhenxi1'} implies that $II>0$.

Therefore, we conclude that $D(\Omega_t)$ is a strictly increasing function for $-\infty<t\le 1$. 

\end{proof}

Combining Proposition \ref{T2} and Proposition \ref{T1}, we obtain Theorem \ref{quadritheorem}. 

\vskip 0.2cm
We note that any quadrilateral must be of either type I (the case in Proposition \ref{T2}) or type II (the case in Proposition \ref{T1}). The following pictures Figure \ref{type1} and Figure \ref{fig:cmd2} illustrate a continuous deformation of an arbitrary quadrilateral into a square, along which $D(\cdot)$ is strictly increasing, as a consequence of Theorem \ref{rhombustheoremc}, Proposition \ref{T2} and Proposition \ref{T1}. This gives another proof of the fact that the square maximizes $D(\cdot)$ among quadrilaterals with a given area, without resorting to the limit process of step-by-step Steiner symmetrizations as used in \cite{BCT}.

\begin{figure}[htp]
	\centering
	\begin{tikzpicture}
		\pgfmathsetmacro\step{0.01}
		\pgfmathsetmacro\transa{5}
		\pgfmathsetmacro\transb{10}
		
		\pgfmathsetmacro\transc{3}
		\pgfmathsetmacro\transd{9}
		
		\pgfmathsetmacro\down{5}
		
		\pgfmathsetmacro\r{sqrt(3/2)}
		\fill (-1,0) circle(0.03) node[left] {\tiny $A^{1}$};
		\fill (0.3,-2) circle(0.03) node[below] {\tiny $B^{1}$};
		\fill (1,0) circle(0.03) node[right] {\tiny $C^{1}$};
		\fill (-0.5,1) circle(0.03) node[above] {\tiny $D^{1}$};
		
		\draw (-1,0)--(-0.5,1)--(1,0)--(0.3,-2)--(-1,0);
		\draw[dotted] (-1,0)--(1,0);
		\draw[gray] (0,-2.5)--(0,1.5);
		\fill[yellow, opacity=0.2] (-1,0)--(-0.5,1)--(1,0)--(0.3,-2)--(-1,0);
		
		\draw[->,thick] (1.7,0)--(\transa-1.7,0);
		\node[below left] at (\transa-2,0) {\small \textit{Step 1}};
        \node[above] at (\transa-2.5,0) {\small \textit{Proposition \ref{T1}}};

		\fill (-1+\transa,0) circle(0.03) node[left] {\tiny $A^{2}$};
		\fill (0.3+\transa,-2) circle(0.03) node[below] {\tiny $B^{2}$};
		\fill (1+\transa,0) circle(0.03) node[right] {\tiny $C^{2}$};
		\fill (0+\transa,1) circle(0.03) node[above] {\tiny $D^{2}$};
		
		\draw[red, ->,thick] (-0.5+\transa,1)--(\transa-\step,1);
		\draw[dashed] (-1+\transa,0)--(-0.5+\transa,1)--(1+\transa,0)--(0.3+\transa,-2)--(-1+\transa,0);
		\draw (-1+\transa,0)--(0+\transa,1)--(1+\transa,0)--(0.3+\transa,-2)--(-1+\transa,0);
		\draw[gray] (0+\transa,-2.5)--(0+\transa,1.5);
		\draw[dotted] (-1+\transa,0)--(1+\transa,0);
		\fill[yellow, opacity=0.2] (-1+\transa,0)--(0+\transa,1)--(1+\transa,0)--(0.3+\transa,-2)--(-1+\transa,0);
		
		\draw[->,thick] ({\transa+1.7},0)--({\transb-1.7},0);
		\node[below] at (\transb-2.5,0) {\small \textit{Step 2}};
        \node[above] at (\transb-2.5,0) {\small \textit{Proposition \ref{T1}}};
		
		\fill (-1+\transb,0) circle(0.03) node[left] {\tiny $A^{3}$};
		\fill (\transb,-2) circle(0.03) node[below] {\tiny $B^{3}$};
		\fill (1+\transb,0) circle(0.03) node[right] {\tiny $C^{3}$};
		\fill (0+\transb,1) circle(0.03) node[above] {\tiny $D^{3}$};
		\draw[->,red,thick] (0.3+\transb,-2)--(\transb+\step,-2);
		
		\draw[dashed] (-1+\transb,0)--(0+\transb,1)--(1+\transb,0)--(0.3+\transb,-2)--(-1+\transb,0);
		\draw[gray] (0+\transb,-2.5)--(0+\transb,1.5);
		\draw[dotted] (-1+\transb,0)--(1+\transb,0);
		\draw (-1+\transb,0)--(0+\transb,1)--(1+\transb,0)--(\transb,-2)--(-1+\transb,0);
		\fill[yellow, opacity=0.2] (-1+\transb,0)--(0+\transb,1)--(1+\transb,0)--(\transb,-2)--(-1+\transb,0);
		
		\fill (-1+\transc,-0.5-\down) circle(0.03) node[left] {\tiny $A^{4}$};
		\fill (\transc,-2-\down) circle(0.03) node[below] {\tiny $B^{4}$};
		\fill (1+\transc,-0.5-\down) circle(0.03) node[right] {\tiny $C^{4}$};
		\fill (0+\transc,1-\down) circle(0.03) node[above] {\tiny $D^{4}$};
		\draw[->,red,thick] (-1+\transc,0-\down)--(-1+\transc,-0.5-\step-\down);
		\draw[->,red,thick] (1+\transc,0-\down)--(1+\transc,-0.5-\step-\down);
		
		\draw[dashed] (-1+\transc,0-\down)--(0+\transc,1-\down)--(1+\transc,0-\down)--(\transc,-2-\down)--(-1+\transc,0-\down);
		\draw (-1+\transc,-0.5-\down)--(0+\transc,1-\down)--(1+\transc,-0.5-\down)--(\transc,-2-\down)--(-1+\transc,-0.5-\down);
		\draw[gray] (0+\transc,-2.5-\down)--(0+\transc,1.5-\down);
		\draw[dotted] (-1+\transc,0-\down)--(1+\transc,0-\down);
		\draw[gray] (-1.2+\transc,-0.5-\down)--(1.2+\transc,-0.5-\down);
		\fill[yellow, opacity=0.2] (-1+\transc,-0.5-\down)--(0+\transc,1-\down)--(1+\transc,-0.5-\down)--(\transc,-2-\down)--(-1+\transc,-0.5-\down);
		
		\draw[->,thick] ({-\transc+3},0-\down)--({-\transc+4.5},0-\down);
		\node[below left] at ({-\transc+4.2},0-\down) {\small \textit{Step 3}};
        \node[above] at ({-\transc+3.7},0-\down) {\small \textit{Proposition \ref{T2}}};
		
		\fill (-\r+\transd, -0.5-\down) circle(0.03) node[left] {\tiny $A^{5}$};
		\fill (0+\transd,-\r-0.5-\down) circle(0.03) node[below] {\tiny $B^{5}$};
		\fill (\r+\transd,-0.5-\down) circle(0.03) node[right] {\tiny $C^{5}$};
		\fill (0+\transd,\r-0.5-\down) circle(0.03) node[above] {\tiny $D^{5}$};
		\draw[->, red,thick] (-1+\transd,-0.5-\down)--(-\r+\transd+\step, -0.5-\down);
		\draw[->,red,thick] (0+\transd,1-\down)--(0+\transd,\r-0.5-\down+\step);
		\draw[->,red,thick] (1+\transd,-0.5-\down)--(\r+\transd,-0.5-\down-\step);
		\draw[->,red,thick] (\transd,-2-\down)--(0+\transd,-\r-0.5-\down);
		
		\draw[dashed] (-1+\transd,-0.5-\down)--(0+\transd,1-\down)--(1+\transd,-0.5-\down)--(\transd,-2-\down)--(-1+\transd,-0.5-\down);
		\draw (-\r+\transd, -0.5-\down)--(0+\transd,\r-0.5-\down)--(\r+\transd,-0.5-\down)--(0+\transd,-\r-0.5-\down)--(-\r+\transd,-0.5-\down);
		\draw[gray] (-1.4+\transd,-0.5-\down)--(1.4+\transd,-0.5-\down);
		\draw[gray] (0+\transd,-2.5-\down)--(0+\transd,1.5-\down);
		\fill[yellow, opacity=0.2] (-\r+\transd, -0.5-\down)--(0+\transd,\r-0.5-\down)--(\r+\transd,-0.5-\down)--(0+\transd,-\r-0.5-\down)--(-\r+\transd,-0.5-\down);
		
		\draw[->,thick] ({\transd-3.5},0-\down)--({\transd-2},0-\down);
		\node[below left] at ({\transd-2.3},0-\down) {\small \textit{Step 4}};
        \node[above] at ({\transd-2.8},0-\down) {\small \textit{Theorem \ref{rhombustheoremc}}};
	
    \end{tikzpicture}

	\caption{Type I: continuous monotonic deformation for $D(\cdot)$}
	\label{type1}
\end{figure}
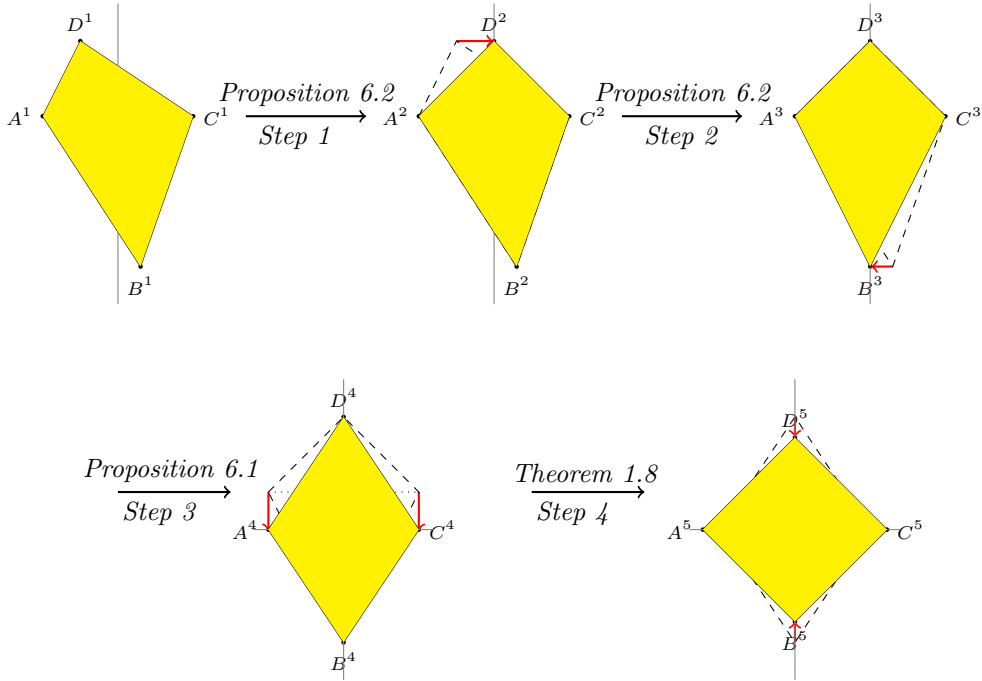

\vskip 0.3cm

\begin{remark}
    In fact, as can be observed from Proposition \ref{T1}, the first and second steps in Figure \ref{type1} can be carried out simultaneously. This establishes the monotonicity of property $D(\cdot)$ under the classical continuous Steiner symmetrization with respect to the direction perpendicular to the diagonal, thereby proving Theorem \ref{quadritheorem}. 
\end{remark}

\begin{remark}
\label{partialsymmetrization}
    Proposition \ref{T1} indeed gives a stronger result: within the Type I case, we can further   decompose the continuous Steiner symmetrization process as a two-step partial continuous symmetrization: continuously symmetrizing one portion and leaving the other one fixed. The partial continuous symmetrization still guarantees the strict monotonicity property of $D(\cdot)$, as illustrated in Figure \ref{type1}.
\end{remark}

\vskip 0.3cm
\begin{figure}[htp]
	\centering
	\begin{tikzpicture}[scale=0.95]
		\pgfmathsetmacro\transa{5}
		\pgfmathsetmacro\transb{10}
		\pgfmathsetmacro\transc{15}
		\pgfmathsetmacro\r{sqrt(3/2)}
		
		\fill (-1,0) circle(0.03) node[left] {\tiny $A^{1}$};
		\fill (-0.3,-2) circle(0.03) node[below] {\tiny $B^{1}$};
		\fill (1,0) circle(0.03) node[right] {\tiny $C^{1}$};
		\fill (-0.5,1) circle(0.03) node[above] {\tiny $D^{1}$};
		
		\draw (-1,0)--(-0.3,-2)--(1,0)--(-0.5,1)--(-1,0);
		\draw[dotted] (-1,0)--(1,0);
		\draw[gray] (0,-2.5)--(0,1.5);
		\fill[yellow, opacity=0.2] (-1,0)--(-0.5,1)--(1,0)--(-0.3,-2)--(-1,0);

		\fill (-1+\transa,0) circle(0.03) node[left] {\tiny $A^{2}$};
		\fill (\transa,-2) circle(0.03) node[below] {\tiny $B^{2}$};
		\fill (1+\transa,0) circle(0.03) node[right] {\tiny $C^{2}$};
		\fill (\transa,1) circle(0.03) node[above] {\tiny $D^{2}$};
		
		\draw[dashed] (-1+\transa,0)--(-0.3+\transa,-2)--(1+\transa,0)--(-0.5+\transa,1)--(-1+\transa,0);
		\draw (-1+\transa,0)--(0+\transa,-2)--(1+\transa,0)--(0+\transa,1)--(-1+\transa,0);
		\draw[gray] (0+\transa,-2.5)--(0+\transa,1.5);
		\draw[gray] (0+\transa-1.5,-0.5)--(0+\transa+1.5,-0.5);
		\fill[yellow, opacity=0.2] (-1+\transa,0)--(0+\transa,-2)--(1+\transa,0)--(0+\transa,1)--(-1+\transa,0);
		\draw[->,red,thick] (-0.3+\transa,-2)--(0+\transa,-2);
		\draw[->,red,thick] (-0.5+\transa,1)--(0+\transa,1);
		\draw[->,thick] (\transa-3.2,0)--(\transa-1.7,0);
		\node[below left] at (\transa-1.7,0) {\small \textit{Step 1}};
        \node[above] at (\transa-2.2,0) {\small \textit{Proposition \ref{T2}}};
		
		\fill (-1+\transb,-0.5) circle(0.03) node[left] {\tiny $A^{3}$};
		\fill (\transb,-2) circle(0.03) node[below] {\tiny $B^{3}$};
		\fill ((1+\transb,-0.5) circle(0.03) node[right] {\tiny $C^{3}$};
		\fill (\transb,1) circle(0.03) node[above] {\tiny $D^{3}$};
		
		\draw[dashed] (-1+\transb,0)--(0+\transb,-2)--(1+\transb,0)--(0+\transb,1)--(-1+\transb,0);
		\draw (-1+\transb,-0.5)--(0+\transb,-2)--(1+\transb,-0.5)--(0+\transb,1)--(-1+\transb,-0.5);
		\draw[gray] (0+\transb-1.5,-0.5)--(0+\transb+1.5,-0.5);
		\draw[gray] (0+\transb,-2.5)--(0+\transb,1.5);
		\fill[yellow, opacity=0.2] (-1+\transb,-0.5)--(0+\transb,-2)--(1+\transb,-0.5)--(0+\transb,1)--(-1+\transb,-0.5);
		\draw[->,red,thick] (-1+\transb,0)--(-1+\transb,-0.5);
		\draw[->,red,thick] (1+\transb,0)--(1+\transb,-0.5);
		\draw[->,thick] (\transb-3.2,0)--(\transb-1.7,0);
		\node[below left] at (\transb-1.7,0) {\small \textit{Step 2}};
        \node[above] at (\transb-2.2,0) {\small \textit{Proposition \ref{T2}}};
		
		\fill (-\r+\transc, -0.5) circle(0.03) node[left] {\tiny $A^{4}$};
		\fill (0+\transc,-\r-0.5) circle(0.03) node[below] {\tiny $B^{4}$};
		\fill (\r+\transc,-0.5) circle(0.03) node[right] {\tiny $C^{4}$};
		\fill (0+\transc,\r-0.5) circle(0.03) node[above] {\tiny $D^{4}$};
		\draw[->, red,thick] (-1+\transc,-0.5)--(-\r+\transc, -0.5);
		\draw[->,red,thick] (0+\transc,1)--(0+\transc,\r-0.5);
		\draw[->,red,thick] (1+\transc,-0.5)--(\r+\transc,-0.5);
		\draw[->,red,thick] (\transc,-2)--(0+\transc,-\r-0.5);
		
		\draw[dashed] (-1+\transc,-0.5)--(0+\transc,-2)--(1+\transc,-0.5)--(0+\transc,1)--(-1+\transc,-0.5);
		\draw (-\r+\transc, -0.5)--(0+\transc,\r-0.5)--(\r+\transc,-0.5)--(0+\transc,-\r-0.5)--(-\r+\transc,-0.5);
		\draw[gray] (-1.4+\transc,-0.5)--(1.4+\transc,-0.5);
		\draw[gray] (0+\transc,-2.5)--(0+\transc,1.5);
		\fill[yellow, opacity=0.2] (-\r+\transc, -0.5)--(0+\transc,\r-0.5)--(\r+\transc,-0.5)--(0+\transc,-\r-0.5)--(-\r+\transc,-0.5);
		\draw[->,thick] (\transc-3.2,0)--(\transc-1.7,0);
		\node[below left] at (\transc-1.7,0) {\small \textit{Step 3}};
        \node[above] at (\transc-2.2,0) {\small \textit{Theorem \ref{rhombustheoremc}}};
	\end{tikzpicture}
	\caption{Type II: continuous monotonic deformation for $D(\cdot)$}
	\label{fig:cmd2}
\end{figure}
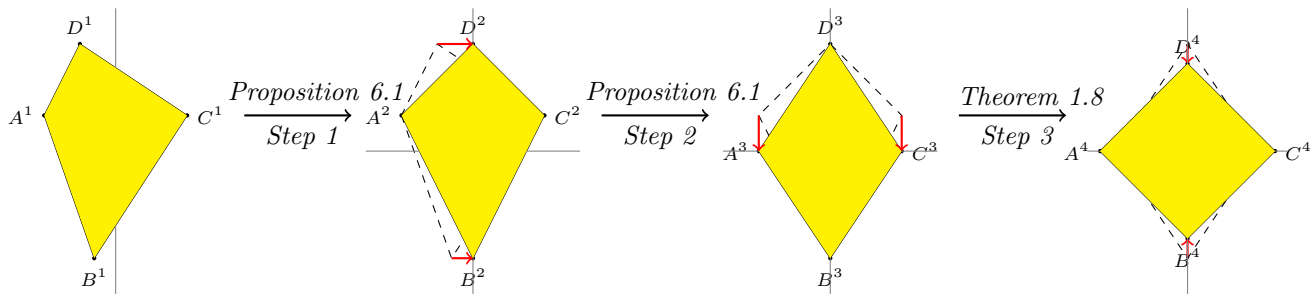

\vskip 0.3cm

\noindent
\textbf{Data Availibility Statement:} This study is a theoretical analysis, and no new data were created or analyzed. Therefore, data sharing is not applicable to this article.

\end{document}